\documentclass[a4paper]{amsart}
\usepackage{amssymb,amsmath} 
\usepackage{mathrsfs} 
\usepackage{enumitem}
\usepackage[capitalize]{cleveref}
\usepackage{xcolor}
\newcounter{alph}

\newtheorem{theo}[alph]{Theorem}
\newcounter{alphp}

\newtheorem{theop}[alphp]{Theorem}
\newcounter{alphpp}

\newtheorem{theopp}[alphpp]{Theorem}
\numberwithin{equation}{section}
\newtheorem*{theadp}{Theorems A' -- D'}
\newtheorem*{theadpp}{Theorems C'' and D''}

\newtheorem{cor}[equation]{Corollary}
\newtheorem{lem}[equation]{Lemma}
\newtheorem{prop}[equation]{Proposition}
\newtheorem{thm}[equation]{Theorem}
\theoremstyle{definition}

\newtheorem{exa}[equation]{Example}

\newtheorem{rem}[equation]{Remark}

\def\N{\mathbb N}
\def\R{\mathbb R}

\def\Z{\mathbb Z}

\def\ve{\varepsilon}
\def\vf{\varphi}
\def\la{\langle}
\def\ra{\rangle}
\def\dv{\rm dv}

\newcommand{\diam}{\operatorname{diam}}

\newcommand{\supp}{\operatorname{supp}}

\begin{document}


\title[Diffusions, random walks, harmonic functions]
{Equivariant discretizations of diffusions, \\ random walks, and harmonic functions}
\author{Werner Ballmann}
\address
{WB: Max Planck Institute for Mathematics, Vivatsgasse 7, 53111 Bonn.}
\email{hwbllmnn\@@mpim-bonn.mpg.de}
\author{Panagiotis Polymerakis}
\address
{PP: Max Planck Institute for Mathematics, Vivatsgasse 7, 53111 Bonn.}
\email{polymerp@@mpim-bonn.mpg.de}

\thanks{We are grateful to Anna Erschler for pointing out a gap in the original version of the manuscript, which most notably led to the reformulation of Theorem \ref{trec}. We would like to thank Fran\c{c}ois Ledrappier, Henrik Matthiesen, Mikhail Zaidenberg,
and the referee for helpful comments.
We would also like to thank the Max Planck Institute for Mathematics and the Hausdorff Center for Mathematics in Bonn for their support and hospitality.}

\date{December 22, 2025}


\begin{abstract}
For covering spaces and properly discontinuous actions with compatible diffusion processes,
we discuss Lyons-Sullivan discretizations of the processes and the associated function theory.
\end{abstract}

\maketitle

\tableofcontents

\section{Introduction}
\label{intro}

We are interested in spaces of bounded or positive $\lambda$-harmonic functions
on Riemannian manifolds and, more generally,
spaces of bounded or positive harmonic functions of diffusion operators.
Our work is motivated by the articles \cite{Su} of Sullivan and \cite{LS} of Lyons and Sullivan.

To a sufficiently large discrete subset $X$ in a connected Riemannian manifold $M$, 
Lyons and Sullivan associate a family of probability measures $\mu=(\mu_y)_{y\in M}$ on $X$
which has a number of important properties.
Among others, Lyons and Sullivan show that
\begin{align}\label{muharm}
  h(y) = \mu_y(h) := \sum_{x\in X}\mu_y(x)h(x)
\end{align}
for all bounded harmonic functions $h$ on $M$ and $y\in M$.
Furthermore, the support of each of the measures $\mu_y$ is all of $X$.
We call $\mu$ the family of \emph{Lyons-Sullivan measures} or, as a shorthand, \emph{$LS$-measures}.
They depend on the choice of data, which we call \emph{$LS$-data},
and these exist if $X$ is \emph{$*$-recurrent} in the terminology of Lyons and Sullivan.
We use the same concept,
but in the context of diffusion operators and in the form used in \cite{BL2}
and develop and extend results from \cite{LS}, \cite{Ka2}, \cite{LZ}, and \cite{BL2} in different directions.

Let $M$ be a connected manifold and $L$ be an elliptic diffusion operator on $M$
that is symmetric on $C^\infty_c(M)$ with respect to a smooth volume element
(see \cref{susdop}).
We say that a smooth function $h$ on $M$ is \emph{$L$-harmonic} if $Lh=0$.

Let $(D_t)_{t\ge0}$ be the diffusion process on $M$ with generator $L$ (see \cref{susdop}).
We say that $M$ is \emph{$L$-complete}, \emph{$L$-recurrent}, or \emph{$L$-transient} if $(D_t)$ is stochastically complete, recurrent, or transient, respectively. 
Note that $M$ is either $L$-recurrent or $L$-transient
and that $L$-recurrence implies $L$-completeness. 

\begin{exa}\label{comrec}
If $M$ is compact, then $M$ is $L$-recurrent. 
\end{exa}

In what follows,
we state our main results in the case of a covering $\pi\colon M\to N$,
with group $\Gamma$ of covering transformations,
where the operator $L$ on $M$ is the pull back of a diffusion operator $L_0$ on $N$.
We have similar results in the case of a countable group, also denoted by $\Gamma$,
acting properly discontinuously and $L$-equivariantly on $M$.
In our final section, we also discuss related results for random walks,
that is, Markov chains on countable sets.

In the literature,
the case of normal Riemannian coverings (with $L,L_0$ the Laplacian) is usually considered.
Recall that $\pi$ is a normal covering if and only if $\Gamma$ acts transitively on the fibers of $\pi$.

We fix a fiber $X$ of $\pi$.
Then $X$ is $*$-recurrent in the sense of Lyons and Sullivan exactly in the case
where $N$ is $L_0$-recurrent,
and the latter is assumed throughout the following.
By \cref{comrec}, the assumption is satisfied whenever the base manifold $N$ is compact.

We let $\mu=(\mu_y)_{y\in M}$ be the $LS$-probability measures on $X$ 
associated to a choice of $LS$-data for $X$.
We say that a function $h$ on $X$ is \emph{$\mu$-harmonic} if it satisfies \eqref{muharm} for all $y\in X$.

\subsection{Main results in the cocompact case}
\label{suspos}
We let $\mathcal{H}^+(M,L)$ and $\mathcal{H}^+(X,\mu)$ be the cones
of positive $L$-harmonic functions on $M$ and positive $\mu$-harmonic functions on $X$,
respectively.
Constant functions on $M$ and $X$ are $L$-harmonic and $\mu$-harmonic, respectively.
If $M$ is $L$-recurrent, then any positive $L$-harmonic function on $M$ is constant.
Similarly, if the $\mu$-random walk on $X$ is recurrent,
then any positive $\mu$-harmonic function on $X$ is constant.

\begin{theo}\label{poha}
Suppose that $N$ is compact
and that $X$ is endowed with the family $\mu$ of $LS$-measures associated to appropriate $LS$-data.
Then
\begin{enumerate}
\item\label{poha1}
for any $h\in\mathcal{H}^+(M,L)$, the restriction $h|_X$ of $h$ to $X$ belongs to $\mathcal{H}^+(X,\mu)$.
More precisely, $h(y)=\mu_y(h)$ for all $y\in M$;
\item\label{poha2}
the restriction map
\begin{align*}
	\mathcal{H}^+(M,L)\to\mathcal{H}^+(X,\mu), \quad h \mapsto h|_X,
\end{align*}
is a $\Gamma$-equivariant isomorphism of cones.
\end{enumerate}
\end{theo}

The meaning of the attribute \lq appropriate\rq\ (here and below) will be specified in \cref{secdico}.
We show in \cref{lsc} that appropriate $LS$-data always exist.

In the case of the Laplacian and with a different choice of $LS$-data,
\cref{poha}.\ref{poha1} is due to Lyons and Sullivan \cite[Theorem 6]{LS}.
\cref{poha}.\ref{poha2} is new even for the Laplacian and is a consequence
of the identification of the minimal parts of Martin boundaries in \cref{mami}.
Note that, in \cite{BL2}, positive $L$-harmonic functions on $M$ that are swept by $F$
(see \eqref{swel4} for this terminology) could be handled.
But in our approach, we can discuss all positive $L$-harmonic functions,
since they can be expressed as integrals of so-called minimal ones.
This is where \cref{mami} comes into play.

Using that our choice of $LS$-data leads to symmetric random walks on $X$
and that positive harmonic functions of irreducible symmetric random walks on nilpotent groups are constant,
\cref{poha} implies \cite[Theorem 1]{LS} which asserts,
in the case where $\pi$ is a normal Riemannian covering with compact base,
that any positive harmonic function on $M$ is constant
if the group of covering transformations of $\pi$ is nilpotent.

In the case where $M$ is $L$-transient,
the Martin boundary $\partial_LM$ and Martin compactification $\mathcal{M}(M,L)=M\cup\partial_LM$ of $M$ are defined
(see \cref{submarb} below).
Similarly, if the $\mu$-random walk on $X$ is transient,
we have the Martin boundary $\partial_\mu X$ and Martin compactification $\mathcal{M}(X,\mu)=X\cup\partial_\mu X$ of $X$.
The \emph{minimal parts} $\partial_L^{\min}M$ and $\partial_\mu^{\min}X$
of the Martin boundaries consist of minimal positive $L$-harmonic
and minimal positive $\mu$-harmonic functions,
respectively.
Here \emph{minimal} refers to the property that these positive $L$-harmonic
respectively $\mu$-harmonic functions dominate only multiples of themselves
and do not dominate any other positive $L$-harmonic respectively $\mu$-harmonic function.
The following result is new even in the case of the Laplacian.

\begin{theo}\label{mami}
Suppose that $N$ is compact, that $M$ is $L$-transient,
and that $X$ is endowed with the family $\mu$ of $LS$-measures associated to appropriate $LS$-data.
Then
\begin{enumerate}
\item\label{mami1}
the $\mu$-random walk on $X$ is transient
and the inclusion $X\to M$ extends to a $\Gamma$-equivariant homeomorphism
\begin{align*}
	\partial_\mu^{\min}X\to \partial_L^{\min}M;
\end{align*}
\item\label{mami2}
$\partial_L^{\min}M=\partial_LM$ if and only if $\partial_\mu^{\min}X=\partial_\mu X$. 
\end{enumerate}
\end{theo}

\cref{mami} follows from \cref{martin},
in which we prove corresponding statements in a more general situation.
The key point is that we can compare sequences of Martin kernels $K(.,y_n)$,
associated to $L$ respectively $\mu$ and converging to minimal positive $L$-harmonic
respectively $\mu$-harmonic functions in the corresponding Martin boundaries,
because of the cocompactness of the covering and the proportionality
of the associated Green functions (established in \cref{Greengreen}).

\subsection{Main results in the recurrent case}
\label{susbou}
In contrast to the cocompact case,
our arguments and results in this subsection
are mostly straightforward extensions of corresponding arguments and results in the literature,
where the case of the Laplacian and normal Riemannian coverings is usually considered.
Nevertheless, there are differences and twists in the setup and argumentation
which may make our exposition more accessible. 

We denote by $\mathcal{H}^\infty(M,L)$ and $\mathcal{H}^\infty(X,\mu)$
the spaces of bounded $L$-harmonic functions on $M$ and bounded $\mu$-harmonic functions on $X$, respectively.
Note that $\mathcal{H}^\infty(M,L)$ vanishes whenever $M$ is $L$-recurrent
and that $\mathcal{H}^\infty(X,\mu)$ vanishes whenever the $\mu$-random walk on $X$ is recurrent.

\begin{theo}\label{boha}
Suppose that $N$ is $L_0$-recurrent
and that $X$ is endowed with the family $\mu$ of $LS$-measures associated to appropriate $LS$-data.
Then
\begin{enumerate}
\item\label{bohar}
for any $h\in\mathcal{H}^\infty(M,L)$,
the restriction $h|_X$ of $h$ to $X$ belongs to $\mathcal{H}^\infty(X,\mu)$;
\item\label{bohap}
for any $h\in\mathcal{H}^+(M,L)$,
either $h$ is $\mu$-harmonic or strictly $\mu$-superharmonic on $X$.
More precisely, either $h(y)=\mu_y(h)$ for all $y\in M$ or $h(y)>\mu_y(h)$ for all $y\in M$.
\item\label{bohab}
the restriction map
\begin{align*}
	\mathcal{H}^\infty(M,L)\to\mathcal{H}^\infty(X,\mu), \quad h \mapsto h|_X,
\end{align*}
is a $\Gamma$-equivariant isomorphism of vector spaces.
\item\label{bohat}
$M$ is $L$-transient if and only if the $\mu$-random walk on $X$ is transient.
\end{enumerate}
\end{theo}

For $h\in\mathcal{H}^+(M,L)$ as in \eqref{bohap}, we say that \emph{$h$ is swept by $F$} or that $h\in\mathcal{H}_F^+(M,L)$, if the first alternative,
$h(y)=\mu_y(h)$ for all $y\in M$ holds.

In the case of the Laplacian,
\eqref{bohar} is part of \cite[Theorem 5]{LS},
\eqref{bohap} is \cite[Theorem 1.10]{BL2},
\eqref{bohab} follows from \cite[Theorem 1.11]{BL2} and
(with other assumptions on the $LS$-data) from \cite[Theorem 1]{Ka2},
and \eqref{bohat} is part of \cite[Theorem 2.7]{BL2}.
In our arguments, we follow \cite{BL2}. 

Lyons and Sullivan also discretize random Brownian paths in $M$ to $\mu$-random sequences in $X$,
see \cite[Section 8]{LS}.
In the proof of \cite[Theorem 1]{Ka2}, their procedure was analyzed by Kaimanovich
and applied (in \cite[Remark 1]{Ka2}) to Poisson boundaries (see \cref{subpoi} below).

\begin{theo}\label{trec}
Suppose that $N$ is $L_0$-recurrent and
and that $X$ is endowed with the family $\mu$ of $LS$-measures associated to appropriate $LS$-data.
Then
\begin{enumerate}
\item\label{recp}
$LS$-path discretization induces a $\Gamma$-equivariant isomorphism
\begin{align*}
  \mathcal{P}(M,L)\to\mathcal{P}(X,\mu)
\end{align*}
of Poisson boundaries.
\item\label{recm}
if $M$ is $L$-transient, restriction $R\colon\partial_L M \cap \bar{X} \to\partial_\mu X$ of Martin kernels from $M$ to $X$ is $\Gamma$-equivariant, continuous, closed, and surjective, where $\bar{X}$ stands for the closure of $X$ in $\mathcal{M}(M,L)$. Furthermore,
\begin{align*}
	R\colon\partial_L M \cap \bar{X} \cap\mathcal H^+_F(M,L) \to \partial_\mu X\cap\mathcal H^+(X,\mu)
\end{align*}
is a $\Gamma$-equivariant homeomorphism.
\end{enumerate}
\end{theo}

In the case of the Laplacian, \eqref{recp} follows from \cite[Theorem 1 and Remark 1]{Ka2}.
Assertion \eqref{recm} corrects the corresponding statements \cite[Theorem 2.8]{BL2} and \cite[Theorem D(2)]{BP}.
The proofs in \cite{BL2,BP} relied on the erroneous assumption that $\partial_\mu X\subseteq\mathcal H^+(X,\mu)$.
However, this does not hold in general, as was pointed out to us by Anna Erschler.
In fact, in a paper in preparation, we obtain counterexamples for certain $LS$-measures in hyperbolic geometries.

\subsection{Main results for normal coverings}
\label{susnor}
The following is an extension of \cite[Theorem, p.\,307]{LS},
where the case of the Laplacian is considered.
Our proof is considerably shorter than the one in \cite{LS},
but we pay for this by using $LS$-discretization.

\begin{theo}\label{dichot}
Suppose that $\pi$ is normal and that $N$ is $L_0$-recurrent.
Let $h$ be a minimal positive $L$-harmonic function on $M$.
Then either $h$ is constant or there is a $\gamma\in\Gamma$ such that $\gamma^*h/h$ is unbounded.
\end{theo}

Using their discretization of Brownian motion,
Lyons and Sullivan showed that, in the case of a normal Riemannian covering $\pi\colon M\to N$,
any bounded harmonic function on $M$ is constant if $N$ is recurrent and $\Gamma$ is $\omega$-hypercentral \cite[Theorem 2]{LS}
(where $\omega$ denotes the first infinite ordinal).
This was generalized by Lin and Zaidenberg to $FC$-hypercentral groups \cite[Corollary 2.6]{LZ} (see \cref{sushy}).
More generally,
they showed that bounded harmonic functions on $M$ are invariant under the $FC$-hypercenter of $\Gamma$
if $N$ is recurrent \cite[Corollary 2.5]{LZ}.

\begin{theo}\label{fchcin}
If $\pi$ is normal and $N$ is $L_0$-recurrent,
then any bounded $L$-harmonic function on $M$ is invariant under the $FC$-hypercenter of $\Gamma$.
In particular, if $\Gamma$ is $FC$-hypercentral,
then any bounded $L$-harmonic function on $M$ is constant.
\end{theo}

Whereas Li and Zaidenberg use the Stone-\v{C}ech compactification of $\Gamma$ in their proof,
we rely on $LS$-discretization.
Nevertheless, our argumentation here draws heavily from \cite{LZ} and the earlier \cite{Li}.

The following result corresponds to \cite[Theorem 3']{LS},
except that the assumption on recurrence is not needed there
(and used here for convenience only).

\begin{theo}\label{noram}
If $\pi$ is normal and $N$ is $L_0$-recurrent,
then there is a $\Gamma$-invariant bounded projection $L^\infty(M)\to\mathcal{H}^\infty(M,L)$.
In particular, if all bounded $L$-harmonic functions on $M$ are constant,
then $\Gamma$ is amenable.
\end{theo}

It seems that the proof in \cite{LS} extends to the situation considered here.
However, we remain in our setup and invoke $LS$-discretization.

\subsection{The motivating example (after Sullivan \cite{Su})}
\label{susexa}
Let $M$ be a connected Riemannian manifold and set
\begin{equation}\label{defbot}
  \lambda_0 = \lambda_0(M) = \inf R(f),
\end{equation}
where $f$ runs through all non-vanishing smooth functions on $M$ with compact support
and $R(f)$ denotes the Rayleigh quotient of $f$.
Recall that $\lambda_0$ is equal to the bottom of the spectrum of the \emph{Friedrichs extension}
of the (geometer's) Laplacian $\Delta$ of $M$,
considered as an unbounded symmetric operator on the space $L^2(M)$ of square-integrable functions on $M$ 
with domain $C^\infty_c(M)$, the space of smooth functions on $M$ with compact support.

It is well known that $\lambda_0$ is the supremum over all $\lambda\in\R$
such that there is a positive $\lambda$-harmonic function $f\colon M\to\R$
(see, e.g., \cite[Theorem 7]{CY}, \cite[Theorem 1]{FS}, or \cite[Theorem 2.1]{Su}).
It is crucial that these $\lambda$-harmonic functions are not required to be square-integrable.
In fact, by the above, $\lambda_0$ is exactly the border between positive and $L^2$-spectrum of $M$.

\begin{rem}\label{rembot}
If the right action of $\pi_1(N,q)$ on $X=\pi^{-1}(q)$ is amenable, then $\lambda_0(M)=\lambda_0(N)$; see \cite{BMP1}.
If the right action of $\pi_1(N,q)$ on $X$ is not amenable and $\lambda_0(N)$ does not belong to the essential spectrum of $N$,
then $\lambda_0(M)>\lambda_0(N)$; see \cite{Po}.
\end{rem}

Let $p=p(t,x,y)$ be the kernel of the minimal heat semigroup on $M$.
Recall that $p(t,x,y)>0$ for all $t>0$ and $x,y\in M$.
Say that $\lambda\in\R$ belongs to the \emph{Green's region of $M$} if
\begin{equation}\label{green}
  \int_0^\infty e^{\lambda t} p(t,x,y) \,dt < \infty
\end{equation}
for some--and then any--pair of points $x\ne y$ in $M$.
By \cite[Theorem 2.6]{Su},
the Green's region of $M$ is either $(-\infty,\lambda_0)$ or   $(-\infty,\lambda_0]$.
In the first case, $M$ is said to be \emph{$\lambda_0$-recurrent}, in the second \emph{$\lambda_0$-transient}.
By \cite[Theorem 2.7]{Su}, if $M$ is $\lambda_0$-recurrent,
then positive $\lambda_0$-harmonic functions are constant multiples of one another.
By \cite[Theorem 2.8]{Su}, if $M$ has a square integrable $\lambda_0$-harmonic function,
then the space of square integrable $\lambda_0$-harmonic functions on $M$
is generated by a square integrable positive $\lambda_0$-harmonic function
and $M$ is $\lambda_0$-recurrent.
In the $\lambda_0$-recurrent case, the associated $\lambda_0$-random motion
with transition densities $e^{\lambda_0t}p(t,x,y)\vf(y)/\vf(x)$
with respect to the Riemannian volume element of $M$,
where $\vf$ is a positive $\lambda_0$-harmonic function on $M$,
does not depend on the choice of $\vf$ (by what was said above) and is recurrent,
by \cite[Theorem 2.10]{Su}).

In the case $\lambda_0=0$, we also speak of harmonic functions and of recurrent or transient manifolds.
In this case, the associated random walk is standard Brownian motion on $M$.

\begin{exa}\label{exafivo}
If $M$ is compact or, more generally, complete and of finite volume, then $\lambda_0=0$,
and constant functions are $\lambda_0$-harmonic and square-integrable.
\end{exa}

\begin{exa}\label{exahype}
(Sullivan \cite[Theorem 2.21]{Su}.)
Let $M=\Gamma\backslash H^m$ be a complete hyperbolic manifold
and suppose that $M$ is geometrically finite, that is,
the action of $\Gamma$ on $H^m$ admits a fundamental domain
with finitely many and totally geodesic sides. 
Then
\begin{align*}
	\lambda_0 =
	\begin{cases}
	(m-1)^2/4 &\text{if $d\le(m-1)/2$,} \\
	d(m-1-d) &\text{if $d\ge(m-1)/2$,}
	\end{cases}
\end{align*}
where $d=d(\Gamma)$ denotes the Hausdorff dimension of the limit set of $\Gamma$.
If $d\ge(m-1)/2$, then $M$ is $\lambda_0$-recurrent.
If $d>(m-1)/2$, then $M$ has square-integrable positive $\lambda_0$-harmonic functions.
(See also \cite[Theorem 2.17 and Corollary 2.18]{Su}.)
\end{exa}

Let $\pi\colon M\to N$ be a Riemannian covering of connected Riemannian manifolds
with group $\Gamma$ of covering transformations
and $X$ be the fiber of $\pi$ over some chosen point in $N$.
Let $\lambda_0=\lambda_0(N)$, and assume that $N$ is $\lambda_0$-recurrent.
Let $\vf$ be the lift of a positive $\lambda_0$-harmonic function from $N$ to $M$
and $\dv$ be the Riemannian volume element of $M$.
Then multiplication $m_\vf$ by $\vf$ gives a unitary transformation $L^2(M,\vf^2\dv)\to L^2(M,\dv)$
which intertwines the diffusion operator $L=Lf=-\Delta f+2\la\nabla\ln\vf,\nabla f\ra$
with the operator $-\Delta+\lambda_0$;
see the computation in the beginning of \cite[Section 8]{Su} (compare also with \cref{susdop} below).
The associated diffusion process has transition densities $e^{\lambda_0t}p(t,x,y)/\vf(x)\vf(y)$
with respect to the volume element $\vf^2\dv$.
Therefore, since $N$ is $\lambda_0$-recurrent, $X$ is $*$-recurrent in the sense of Lyons and Sullivan.
Hence the above results apply if $N$ is compact (then $\lambda_0=0$) or $\lambda_0$-recurrent, respectively.

\subsection{Discussion}
\label{susdis}
An important consequence of our results is that, under specific conditions,
problems about $L$-harmonic functions on $M$,
minimal Martin boundaries $\partial_L^{\min}M$, and Poisson boundaries $\mathcal{P}(M,L)$
are equivalent to problems about $\mu$-harmonic functions on $X$,
minimal Martin boundaries $\partial_\mu^{\min}X$, and Poisson boundaries $\mathcal{P}(X,\mu)$.
The most basic problems in this direction are
the non-existence of non-constant bounded or positive $L$-harmonic functions,
and correspondingly for random walks.
These properties are referred to as the \emph{Liouville} and \emph{strong Liouville property}, respectively.

Let $\pi\colon M\to N$ be a normal Riemannian covering  of connected manifolds
with group $\Gamma$ of covering transformations.
Lyons and Sullivan showed that $\Gamma$ is amenable if $M$ has the Liouville property,
that is, if $\mathcal{P}(M,\Delta)$ is trivial.
The converse does not hold.
In fact, Lyons and Sullivan gave examples of Riemannian covers $M$ of closed surfaces $N$,
where $\Gamma$ is two-step solvable and $\mathcal{P}(M,\Delta)$ is nontrivial \cite[pp. 299-300]{LS}.
In the same direction, Erschler showed that $\mathcal{P}(M,\Delta)$ is nontrivial
if $N$ is closed and $\Gamma$ is a Baumslag group $B_d$ with $d\ge3$ \cite[Theorem 5.2]{Er},
a two-step solvable and finitely presentable group.
In her proof, Erschler uses the isomorphism $\mathcal{P}(M,\Delta)\cong\mathcal{P}(X,\mu)$,
where $X$ is a fibre of $\pi$ and $\mu$ an associated family of Lyons-Sullivan measures on $X$.
Note also that, since $B_d$ is finitely presentable,
$N$ can be chosen to have fundamental group $\Gamma=B_d$ and $M$ to be the universal covering space of $N$.

At this point, it is not clear whether there are any non-trivial conditions on $\Gamma$
which are equivalent to the Liouville property of $M$ if $N$ is closed or recurrent.
As for sufficient conditions,
Lyons and Sullivan proved that $\mathcal{P}(M,\Delta)$ is trivial
in the case where $N$ is recurrent and $\Gamma$ is $\omega$-hypercentral \cite[Theorem 2]{LS}. 
This was extended by Lin and Zaidenberg to the case where $\Gamma$ is $FC$-hypercentral \cite[Corollary 2.6]{LZ}.
Kaimanovich showed by analytic methods that $\mathcal{P}(M,\Delta)$ is trivial if $N$ is closed
and $\Gamma$ is polycyclic or of subexponential growth \cite[Theorem 8 and Corollary 1]{Ka1}.
All the above results can also be shown for elliptic diffusion operators
by using the isomorphisms $\mathcal{H}^\infty(M,L)\cong\mathcal{H}^\infty(X,\mu)$
and between $\mathcal{P}(M,\Delta)\cong\mathcal{P}(X,\mu)$, respectively.
In the present article,
(the second part of) \cref{fchcin} is an example where this strategy applies.
Another one is the consequence of \cref{noram} that $M$ admits
positive $\lambda_0(N)$-harmonic functions which are not $\Gamma$-invariant,
provided that $N$ is $\lambda_0$-recurrent and $\Gamma$ is not amenable.
In the context of the Liouville property,
the entropy criterion of Kaimanovich and Vershik for random walks on groups is also of interest \cite[Theorem 1.1]{KV}.

The strong Liouville property, that is, the triviality of $\partial_\Delta^{\min}M$,
was obtained by a direct argument by Lyons and Sullivan
in the case where $N$ is closed and $\Gamma$ is nilpotent  \cite[Theorem 1]{LS}.
As we mentioned further up,
the extension of this result to elliptic diffusion operators is also an immediate consequence
of the isomorphism $\mathcal{H}^+(M,L)\cong\mathcal{H}^+(X,\mu)$ and the corresponding result for nilpotent groups (see \cite{CD,Ma} and also \cref{nil2}).
Lyons and Sullivan point out the example of a rank two Abelian covering space of the two-sphere
with four points removed which admits non-constant positive harmonic functions,
thus showing that the assumption that $N$ is closed is essential.
This example was generalized and systematized by Epstein to Abelian covers of surfaces of finite type \cite[Theorem 3]{Ep}.

For the case where $N$ is closed,
Lyons and Sullivan posed the problem whether $\Gamma$ has subexponential growth
if and only if $M$ has the strong Liouville property \cite[p.\,305]{LS}.
Bougerol and Elie obtained the converse direction of this under the assumption that $\Gamma$ is linear \cite[Theorem 1.6]{BE} and also \cite[Theorem 1.1]{BBE}.
More recently, the second author of this article proved the converse direction in full generality \cite{Po2} (see also the erratum to appear, or the corrected arXiv version).

By the well known theorem of Anderson and Schoen, $\partial_\Delta M$ is naturally isomorphic to the sphere at infinity
and $\partial_\Delta^{\min}M=\partial_\Delta M$ 
if $M$ is a complete and simply connected Riemannian manifold with pinched negative sectional curvature
\cite{AS,An,Ki}.
In the case where $M$ is the universal covering space of a closed rank one manifold $N$ of non-positive curvature,
the Poisson boundary $\mathcal{P}(M,\Delta)$ is naturally isomorphic to the sphere at infinity of $M$,
using again the isomorphism between $\mathcal{P}(M,\Delta)$ and $\mathcal{P}(X,\mu)$, where $X$ is a fibre of $\pi\colon M\to N$ and $\mu$ an associated family of Lyons-Sullivan measures on $X$ \cite{BL1}.
However, the relation between the Martin boundary $\partial_\Delta M$ and the sphere at infinity is unclear.
One hope in this direction, nurtured by the isomorphism $\partial_\Delta^{\min}M=\partial_\mu^{\min}X$,
is that the latter might be more accessible than $\partial_\Delta^{\min}M$.

\subsection{Structure of the article}
\label{sustru}
Besides other preliminaries, \cref{secpre} contains short descriptions of Martin and Poisson boundaries.
In \cref{seclsd}, we describe the $LS$-discretization procedure and its properties.
Much of this is a translation of known results about the Laplacian and Brownian motion.
However, \cref{martin} and \cref{harmc} about Martin boundaries and positive harmonic functions
are also new for the Laplacian and Brownian motion.
In \cref{secdico},
we apply the results of \cref{seclsd} to covering projections and prove Theorems \ref{poha} -- \ref{trec}.
The general \cref{liftp} and \cref{liftpc} about diffusions on covering spaces might be useful in other contexts.
In \cref{secdiac}, we get analogs of Theorems \ref{poha} -- \ref{trec}
and extensions of Theorem \ref{dichot} -- \ref{noram} for properly discontinuous group actions.
In \cref{secrwg}, we obtain analogs of Theorems \ref{poha} -- \ref{noram} for random walks,
that is, Markov chains on countable sets. 

\section{Preliminaries}
\label{secpre}

\subsection{Diffusion operators and processes}
\label{susdop}
Let $M$ be a connected manifold and $L$ be an elliptic diffusion operator on $M$,
in coordinates of $M$ written as
\begin{align}\label{dol}
	Lf = \frac12 \sum_{ij} a^{ij}\frac{\partial^2f}{\partial x^i\partial x^j} + \sum_k b^k\frac{\partial f}{\partial x^k},
\end{align}
where the matrix $(a^{ij})$ is symmetric and positive definite.
A function $h$ on $M$ is said to be \emph{$L$-harmonic} if $Lh=0$.

The inverse of the coefficient matrix $(a^{ij})$
is the fundamental matrix of a well defined Riemannian metric on $M$,
the \emph{Riemannian metric associated to $L$}.
Its Laplacian $\Delta$ has twice the principal symbol of $-L$,
and hence $L+\frac12\Delta$ is of first order.
Since $L$ and $\Delta$ vanish on constant functions, $L+\frac12\Delta$ is a vector field.

We assume that $L$ is symmetric on $C^\infty_c(M)$ with respect to a smooth volume element on $M$,
that is, a measure on $M$ of the form $\vf^2\dv$, where $\vf$ is a positive smooth function on $M$ and $\dv$ the volume element of the Riemannian metric on $M$ associated to $L$.
The symmetry of $L$ then implies that
\begin{align}\label{doldel}
	Lf = -\frac12\Delta f + \la\nabla\ln\vf,\nabla f\ra.
\end{align}
Conversely, any operator of this from is an elliptic diffusion operator on $M$
which is symmetric on $C^\infty_c(M)$ with respect to the volume element $\vf^2\dv$.

Because of \eqref{doldel}, the diffusion process $(D_t)$ associated to $L$ is Brownian motion with drift $\nabla\ln\vf$.
We assume throughout that $(D_t)$ has infinite life time since this is a consequence of the assumptions in our applications, namely the existence of recurrent subsets.
We will encounter different models of $(D_t)$,
but usually we view it as defined on the space $\Omega$ of continuous paths $\omega\colon\R_{\ge0}\to M$
with $D_t(\omega)=\omega(t)$
and associated family of probability measures $(P_x)_{x\in M}$ on $\Omega$.
We have $P_x[\omega(0)=x]=1$.
Moreover, the distribution of $P_x$ at time $t>0$ has a smooth density $p_{t,x}=p_t(x,y)$ with respect to $\vf^2\dv$
which is symmetric in $x,y$.
For a measure $\mu$ on $M$, we set
\begin{equation}\label{pm}
	P_\mu = \int_{M} \mu(dx)P_x.
\end{equation}
Recall that $(D_t)$ is a strong Markov process and that, for any $f\in C^\infty_c(M)$,
\begin{align}\label{fx}
	f\circ D_t - f\circ D_0 - \int_0^t((Lf)\circ D_s)ds
\end{align}
is a martingale.

\subsection{Martin boundary}
\label{submarb}
The Martin boundary is defined when the diffusion process is transient,
that is, when its Green function $G(x,y)<\infty$ for all $x\ne y$ in $M$.
Recall that $G(.,y)$ is an $L$-superharmonic function which is $L$-harmonic on $M\setminus\{y\}$.
Fix $x_0\in M$ and define the \emph{Martin kernel} functions $K(.,y)$ by
\begin{align}\label{marker}
  K(x,y) = \frac{G(x,y)}{G(x_0,y)}.
\end{align}
A sequence $(y_n)$ in $M$ is said to \emph{converge to a Martin boundary point} $\eta$ of $M$
if $d(x_0,y_n)\to\infty$ and if $K(.,y_n)$ converges (pointwise) on $M$.
In this case, the limit function is $L$-harmonic, since the convergence is actually locally smooth.
The limit function is then denoted by $K(.,\eta)$ and is identified with $\eta$,
and the space of limit functions is denoted by $\partial_LM$.
The Martin compactification $\mathcal{M}(M,L)=M\cup\partial_LM$,
where points $y\in M$ are identified with the Martin kernel $K(.,y)$.
Then $\mathcal{M}(M,L)$ together with pointwise convergence is a compact Hausdorff space with $M$ homeomorphically embedded.

Note that $K(.,\eta)$ is positive with $K(x_0,\eta)=1$.
Any minimal positive harmonic function is a Martin boundary point.
Moreover, for any positive $L$-harmonic function $h$ on $M$ with $h(x_0)=1$,
there is a unique probability measure $\nu_h$ on the minimal part $\partial_L^{\min}M$ of the Martin boundary $\partial_LM$,
the part of $\partial_LM$ consisting of minimal positive $L$-harmonic functions,
such that
\begin{align}\label{poisson2}
	h(x) = \int_{\partial_L^{\min}M }K(x,\eta)\nu_h(d\eta).
\end{align}
We will also need Martin boundaries $\partial_\mu X$ of random walks on countable sets $X$
with family $\mu=(\mu_x)_{x\in X}$ of transition probabilities.
The definition and discussion of these is analogous to the above (see \cite[Chapter 7]{Wo} for instance), except that, in general,
the limit functions of Martin kernels are $\mu$-superharmonic and not necessarily $\mu$-harmonic.
In other words, the Martin boundary $\partial_\mu X$ may contain $\mu$-superharmonic functions which are not $\mu$-harmonic.
However, minimal positive $\mu$-harmonic functions belong to $\partial_\mu X$
and constitute the minimal Martin boundary $\partial_\mu^{\min} X$.
Furthermore,  similarly to (\ref{poisson2}),
there is a representation formula for positive $\mu$-harmonic functions in terms of minimal positive $\mu$-harmonic functions.

\subsection{Poisson boundary}
\label{subpoi}
Consider a random walk on a countable set $X$
with family $\mu=(\mu_x)_{x\in X}$ of transition probabilities and path space $\Omega=X^{\N_0}$.
Assume that the random walk is \emph{irreducible}, that is,
for all $x,y$ in $X$, there exist $x_1,\dots,x_k$ in $X$ such that
\begin{align}\label{mui}
  \mu_x(x_1)\mu_{x_1}(x_2)\dots\mu_{x_{k-1}}(x_k)\mu_{x_k}(y) > 0.
\end{align}
For $k\ge2$,
define a family $\mu^k=(\mu^k_x)_{x\in X}$ of probability measures on $X$ recursively by
\begin{align}\label{muk}
  \mu^k_x(y) = \sum_{z\in X}\mu_x(z)\mu^{k-1}_z(y),
\end{align}
where $\mu^1=\mu$ (and $\mu^0=\delta=(\delta_x)_{x\in X}$, the family of Dirac measures on $X$).
Irreducibility is then equivalent to the property that, for all $x,y\in X$,
\begin{align}\label{mui2}
	\mu^k_x(y)>0
\end{align}
for some $k\ge1$.

Endow $X$ with the discrete topology
and let $\mathcal{B}_\Omega$ be the $\sigma$-algebra of Borel sets of $\Omega$.
For $x\in X$,
let $P_x$ be the probability measure on $(\Omega,\mathcal{B}_\Omega)$ with
\begin{align*}
	&P_x[\omega(0)=x_0,\omega(k_1)\in A_1,\dots,\omega(k_n)\in A_n] \\
	&\hspace{10mm}
	= \delta_x(x_0)\sum_{x_1\in A_1}\mu^{k_1}_{x_0}(x_1)\sum_{x_2\in A_2}\mu^{k_2-k_1}_{x_1}(x_2)
		\cdots\sum_{x_n\in A_n}\mu^{k_n-k_{n-1}}_{x_{n-1}}(x_n)
\end{align*}
for all integers $0<k_1<k_2<\dots<k_n$ and subsets $A_1,\dots,A_n$ of $X$.
Let $\mathcal{F}$ be the extension of $\mathcal{B}_\Omega$ by subsets $N$ of $\Omega$
which are null sets with respect to all $P_x$.
Then the measures $P_x$ extend naturally to $(\Omega,\mathcal{F})$.

We say that $\omega_1,\omega_2\in\Omega$ are \emph{stationary equivalent} if there are $k_1,k_2\in\N$
such that $\omega_1(k_1+k)=\omega_2(k_2+k)$ for all $k\in\N$.
We say that a measurable subset $A$ of $\Omega$ is \emph{stationary mod zero}
if there are sets $N^-\subseteq A$ and $N^+\subseteq\Omega\setminus A$,
which are null sets with respect to all $P_x$,
such that $(A\setminus N^-)\cup N^+$ is the union of stationary equivalence classes.
The set $\mathcal{S}$ of measurable subsets of $\Omega$ which are stationary mod zero
is a $\sigma$-subalgebra of $\mathcal{F}$.
By the irreducibility of the random walk, all the measures $P_x$ are equivalent on $\mathcal{S}$.
We call $\Omega$ together with $\mathcal{S}$ and the restrictions of the measures $P_x$ to $\mathcal{S}$
a \emph{Poisson boundary} of the Markov chain.
More generally, if $\pi \colon (\Omega,\mathcal{S}) \to (\mathcal{P},\mathcal{T})$ is a surjection
and $\mathcal{P}$ is equipped with the measures $\nu_x=\pi_*P_x$,
then we call $(\mathcal{P},\mathcal{T},(\nu_x)_{x\in X})$ a \emph{Poisson boundary} of the Markov chain
if $\pi^*\colon L^\infty(\mathcal{P},\mathcal{T})\to L^\infty(\Omega,\mathcal{S})$ is an isomorphism.
(In other words, given the other conditions,
$\pi^*\colon \mathcal{T}\to\mathcal{S}$ is an isomorphism modulo sets of measure zero.)

A bounded function $f\colon X\to\R$ is called \emph{$\mu$-harmonic} if
\begin{align}\label{pharm}
	f(x) = \sum_{y\in X} \mu_x(y)f(y)
\end{align}
for all $x\in X$.
For any such function $f$, the family $(f\circ X_k)_{k\in\N_0}$ is a bounded martingale,
where $X_k(\omega)=\omega_k$.
Therefore, for each $x\in X$,
\begin{align*}
	f_\infty(\omega) = \lim_{k\to\infty}f(X_k(\omega))
\end{align*}
exists for $P_x$-almost every $\omega\in\Omega$.
If the limit exists for $\omega$, then also for any $\omega'$ which is stationary equivalent to $\omega$,
and then $f_\infty(\omega)=f_\infty(\omega')$.
Hence the limit function $f_\infty$ is $\mathcal{S}$-measurable.
Since all $P_x$ are equivalent on $\mathcal{S}$,
$f_\infty$ is uniquely defined as an element of $L^\infty(\Omega,\mathcal{S})$.
Moreover, since $(f\circ X_k)_{k\in\N_0}$ is a bounded martingale,
\begin{align*}
	f(x) =  \int_\Omega f_\infty(\omega) P_x(d\omega).
\end{align*}
More generally, for any Poisson boundary as above, we get
\begin{align}\label{poisson1}
	f(x) = \int_\mathcal{P} \vf(\eta) \nu_x(d\eta)
\end{align}
where $\vf=f_\infty\circ\pi$.
Conversely, any $f$ defined in this way is a bounded $\mu$-harmonic function on $X$.
The above formula is similar to the Poisson formula \eqref{poisson2},
except that here the measures are given, whereas there the kernel functions are given.

\subsection{$FC$-hypercenters and $FC$-hypercentral groups}
\label{sushy}
For any group $\Gamma$,
let $FC(\Gamma)$ be the union of all elements of $\Gamma$ whose conjugacy classes are finite,
the \emph{$FC$-center of $\Gamma$}, a normal subgroup of $\Gamma$ containing the center of $\Gamma$.
The \emph{upper $FC$-central series} of normal subgroups $\Gamma_\alpha$ of $\Gamma$ is defined by
\begin{alignat}{3}\label{fchc}
  \Gamma_0 & =\{e\}, \notag \\
  \Gamma_\alpha/\Gamma_\beta & = FC(\Gamma/\Gamma_\beta) & &\quad\text{if $\alpha=\beta+1$,} \\
  \Gamma_\alpha &= \cup_{\beta<\alpha}\Gamma_\beta & &\quad\text{if $\alpha$ is a limit ordinal}, \notag
\end{alignat}
where $\alpha$ runs over all ordinals.
The series of $\Gamma_\alpha$ stabilizes eventually,
that is, there is a first ordinal $\alpha$ with $\Gamma_\alpha=\Gamma_\beta=:\Gamma_{\lim}$ for all $\beta\ge\alpha$,
and $\Gamma_{\lim}$ is called the \emph{$FC$-hypercenter of $\Gamma$}.
We say that $\Gamma$ is \emph{$FC$-$\alpha$-hypercentral} if $\Gamma=\Gamma_{\alpha}$.
Replacing the $FC$ centers by centers,
we get the corresponding and more common notions of \emph{upper central series},
\emph{hypercenter}, and \emph{$\alpha$-hypercentral}.
Clearly, $\alpha$-hypercentral groups are $FC$-$\alpha$-hypercentral.
If the specific $\alpha$ is irrelevant,
we also speak of \emph{$FC$-hypercentral} or \emph{hypercentral} groups, respectively.
Synonyms in the literature are \emph{$FC$-hypernilpotent} and \emph{hypernilpotent}. 

Echterhoff showed that $\Gamma$ is $FC$-hypercentral if and only if
it is amenable and every prime ideal of its group $C^*$-algebra is maximal \cite{Ec},
thus providing one reason to non-experts why such groups are interesting.
On the other hand, if $\Gamma$ is finitely generated,
then $\Gamma$ is $FC$-hypercentral if and only if it is virtually nilpotent.
 
\section{Lyons-Sullivan discretization}
\label{seclsd}
Following earlier work of Furstenberg,
Lyons and Sullivan (LS) constructed a discretization of Brownian motion on Riemannian manifolds \cite{LS}.
The LS-construction actually applies to diffusion processes,
and this is one of the objectives of this section.
Our presentation of the LS-construction is close to the one in \cite{BL2},
where the original case of Brownian motion is discussed.

\subsection{Balayage and $L$-harmonic functions}
\label{subset}
Let $F\subseteq M$ be closed and $V\subseteq M$ be open.
For $\omega\in\Omega$, the respective \emph{hitting} and \emph{exit time},
\begin{equation}\label{rfsv}
\begin{split}
  R^F(\omega) &= \inf \{ t\ge0 \mid \omega(t)\in F \}, \\
  S^V(\omega) &= \inf \{ t\ge0 \mid \omega(t) \in M\setminus V\},
\end{split}
\end{equation}
are stopping times.
For a Borel subset $A\subseteq M$, let
\begin{equation}
\begin{split}
  \beta(\mu,F)(A) &= \beta_\mu^F(A) = P_\mu(\omega(R^F(\omega))\in A), \\
  \ve(\mu,V)(A) &= \ve_\mu^V(A) = P_\mu(\omega(S^V(\omega))\in A),
\end{split}
\end{equation}
where $\beta$ stands for \emph{balayage} and $\ve$ for \emph{exit}.
In the case of Dirac measures, $\mu=\delta_x$, we use the shorthand $x$ for $\delta_x$.
If $\beta(x,F)(F)=1$ for all $x\in M$, then $F$ is said to be \emph{recurrent}.
This is equivalent to $R^F<\infty$ almost surely with respect to each $P_x$.

\begin{prop}\label{swel}
Let $F$ be a recurrent closed subset of $M$, $\mu$ a finite measure on $M$, and $h\colon M\to\R$ an $L$-harmonic function.
Then we have:
\begin{enumerate}
\item\label{swelb}
If $h$ is bounded, then $\mu(h)=\beta(\mu,F)(h)$.
\item\label{swelp}
If $h$ is positive, then $\beta(\mu,F)(h)\le\mu(h)$.
\end{enumerate}
\end{prop}

\begin{proof}
It suffices to consider the case where $\mu=\delta_x$.
For $x\in F$, we have $\beta(x,F)=\delta_x$ and hence that $\beta(x,F)(h)=h(x)$.
Let now $x\in M\setminus V$ and
\begin{align*}
	x \in U_1 \subseteq U_2 \subseteq U_3 \subseteq \dots
\end{align*}
be a sequence of relatively compact open subsets exhausting $M\setminus F$.
Let $\chi_n$ be a cut-off function on $M$ with $U_n\subseteq\{\chi_n=1\}\subseteq\supp\chi_n\subseteq U_{n+1}$ and set $f_n=\chi_nh$.
Then $f_n\in C^\infty_c(M)$ and hence
\begin{align*}
	Y_t = f_n\circ X_t - f_n(x) - \int_0^t ((Lf_n)\circ X_s) ds 
\end{align*}
is a $P_x$-martingale, hence also $Y_{t\wedge\sigma_n}$, where $\sigma_n$ denotes the exit time from $U_n$.
Since $f_n$ and $h$ coincide on $U_n$ and $h$ is $L$-harmonic,
we get that $h\circ X_{t\wedge\sigma_n}$ is a $P_x$-martingale and also that
$h\circ X_{t\wedge\sigma_n} = E_x[h\circ X_{\sigma_n}|\mathcal{F}_{t\wedge\sigma_n}]$.
In particular, $h(x)=E_x[h\circ X_{\sigma_n}]$.
Now we have the increasing limit $\sigma_n\to\tau$, in particular the convergence $X_{\sigma_n(\omega)}(\omega)\to X_{\tau(\omega)}(\omega)$,
where $\tau$ denotes the time of hitting $F$.
In the case where $h$ is bounded, the dominated convergence theorem implies
that $h(x)=E_x[h\circ X_{\tau}]$.
In the case where $h$ is positive, the Fatou lemma gives
\begin{align*}
	h(x) = E_x[h\circ X_{\sigma_n}]
	&= \liminf E_x[h\circ X_{\sigma_n}] \\
	&\ge  E_x[\liminf(h\circ X_{\sigma_n})]
	= E_x[h\circ X_{\tau}],
\end{align*}
where we use that $(h\circ X_{\sigma_n(\omega)})(\omega)\to(h\circ X_{\tau(\omega)})(\omega)$.
\end{proof}

\begin{cor}\label{swel2}
If the diffusion process on $M$ is recurrent,
then positive $L$-harmonic functions on $M$ are constant.
\end{cor}

\begin{proof}
If the diffusion process is recurrent, then any closed ball in $M$ is recurrent.
By \cref{swel},
a minimum of a positive harmonic function $h$ on such a ball is a global minimum of $h$,
and hence $h$ is constant, by the maximum principle.
\end{proof}

\begin{rem}\label{swel4}
The analogue of \cref{swel2} also holds for random walks.
\end{rem}

An $L$-harmonic function $h$ on $M$ is said to be \emph{swept by $F$}
if $\beta(x,F)(h)=h(x)$ for all $x\in M$.
Then
\begin{align}\label{swee}
	\mu(h)=\beta(\mu,F)(h).
\end{align}
for all finite measures $\mu$ on $M$.
By \cref{swel}.\ref{swelb}, any bounded $L$-harmonic function is swept by any recurrent closed subset of $M$.

\subsection{$LS$-discretization and $L$-harmonic functions}
\label{sublsdh}
Let $X$ be a discrete subset of $M$.
Families $(F_x)_{x\in X}$ of compact subsets and $(V_x)_{x\in X}$ of relatively compact open subsets of $M$
together with a constant $C>1$  will be called \emph{regular Lyons-Sullivan data for $X$}
or, for short,  \emph{LS-data for $X$} if
\begin{enumerate}[label=(D\arabic*)]
\item\label{d1}
$x\in\mathring{F}_x$ and $F_x\subseteq V_x$ for all $x\in X$;
\item\label{d2}
$F_x\cap V_y=\emptyset$ for all $x\ne y$ in $X$;
\item\label{d3}
$F=\cup_{x\in X}F_x$ is closed and recurrent;
\item\label{d4}
for all $x\in X$ and $y\in F_x$,
\begin{equation*}
  \frac{1}{C} < \frac{d\ve(y,V_x)}{d\ve(x,V_x)} < C.
\end{equation*}
\end{enumerate}
We say that $X$ is \emph{$*$-recurrent} if it admits LS-data.
Our requirements  \ref{d1} and \ref{d2} are more restrictive than the corresponding ones in \cite{LS}
and are conform to the ones in \cite{BL2}.

Suppose now that we are given LS-data as above.
For a finite measure $\mu$ on $M$, define measures
\begin{align}\label{lsmm}
	\mu' = \sum_{x\in X} \int_{F_x}\beta_\mu^F(dy)(\ve_y^{V_x}-\frac1C\ve_x^{V_x})
	\quad\text{and}\quad
	\mu'' = \frac1C\sum_{x\in X} \int_{F_x}\beta_\mu^F(dy)\delta_x
\end{align}
on $M$ with support on $\cup_{x\in X}\partial V_x$ and $X$, respectively.

\begin{prop}\label{lsh1}
If $h$ is a positive $L$-harmonic function on $M$ swept by $F$ and $\mu$ is a finite measure on $M$, then
\[\mu(h)=\mu'(h)+\mu''(h) \quad\text{and}\quad \mu'(h)\le(1-\frac1{C^2})\mu(h).\]
\end{prop}

\begin{proof}
For the first assertion, we compute
\begin{align*}
	\mu'(h)
	&= \sum_{x\in X} \int_{F_x}\beta_\mu^F(dy)(\ve_y^{V_x}(h)-\frac1C\ve_x^{V_x}(h)) \\
	&= \sum_{x\in X} \int_{F_x}\beta_\mu^F(dy)(h(y)-\frac1Ch(x)) \\
	&= \int_{F}\beta_\mu^F(dy)h(y) - \mu''(h) \\
	&= \beta_\mu^F(h) - \mu''(h)
	= \mu(h) - \mu''(h).
\end{align*}
Moreover, by \ref{d4}, we have
\begin{align*}
	\ve_y^{V_x}(h)-\frac1C\ve_x^{V_x}(h)
	\le (1-\frac1{C^2})\ve_y^{V_x}(h)
\end{align*}
for all $x\in X$ and $y\in F_x$.
This implies the second assertion.
\end{proof}

For $y\in M$, let now
\begin{equation}\label{mm2}
	\mu_{y,0}=
	\begin{cases}
	\delta_y &\text{if $y\notin X$,} \\
	\ve(y,V_y) &\text{if $y\in X$,} 
\end{cases}
\end{equation}
and set recursively, for $n\ge1$,
\begin{equation}\label{mmn}
  \mu_{y,n} = (\mu_{y,n-1})' \quad\text{and}\quad \tau_{y,n} = (\mu_{y,n-1})''.
\end{equation}
The associated \emph{LS-measure} is the probability measure
\begin{equation}\label{lsm}
  \mu_y = \sum_{n\ge1} \tau_{y,n}
\end{equation}
with support on $X$.

\begin{prop}\label{lsm2}
The LS-measures $\mu_y$ have the following properties:
\begin{enumerate}
\item\label{lsm2a}
$\mu_y$ is a probability measure on $X$ such that $\mu_y(x)>0$ for all $x\in X$;
\item\label{lsm2b}
for any $x\in X$ and diffeomorphism $\gamma$ of $M$ leaving $L$, $X$, and the LS-data invariant,
\begin{equation*}
  \mu_{\gamma y}(\gamma x) = \mu_y(x);
\end{equation*}
\item\label{lsm2c}
for all $x\in X$,
\begin{equation*}
  \mu_x = \int_{\partial V_x} \ve_x^{V_x}(dy)\mu_y;
\end{equation*}
\item\label{lsm2d}
for all $x\in X$ and $y\in F_x$ different from $x$,
\begin{equation*}
  \mu_y =  \frac1C\delta_x + \int_{\partial V_x}\ve_x^{V_x}(dz) (\frac{d\ve(y,V_x)}{d\ve(x,V_x)}-\frac1C)\mu_z;
\end{equation*}
\item\label{lsm2e}
for any $y\in M\setminus F$ and stopping time $T\le R^F$,
\begin{equation*}
  \mu_y = \int \pi_y^T(dz)\mu_z,
\end{equation*}
where $\pi_y^T$ denotes the distribution of $P_y$ at time $T$.
\end{enumerate}
\end{prop}

\begin{proof}
For the total mass of the $\mu_{y,n}$,
we have $\mu_{y,0}(M)=1$ and, recursively,
\begin{align*}
	\mu_{y,n}(M) \le (1-\frac1{C^2})\mu_{y,n-1}(M) \le (1-\frac1{C^2})^n\mu_{y,0}(M).
\end{align*}
Since
\begin{align*}
	\mu_{y,0}(M) = \mu_{y,n}(M) + \sum_{1\le k\le n}\tau_{y,k}(M),
\end{align*}
we get that $\mu_y$ inherits all the mass of $\mu_{y,0}$ eventually.
This shows the first claim in \eqref{lsm2a}.
The second is clear since $P_y[\{\omega\in\Omega\mid\text{$\omega(t)\in F_x$ for some $t>0$}\}]>0$ for all $x\in X$.
Assertion \eqref{lsm2b} is obvious.
Assertions \eqref{lsm2c}--\eqref{lsm2e} follow immediately from the definition of the LS-measures and the strong Markov property of the process,
observing that $\beta_y^F=\delta_y$ in  \eqref{lsm2d}.
\end{proof}

In the case of Brownian motion, the following is Theorem 1.10 of \cite{BL2}
(where the reader is referred to the discussion on \cite[p.\,317]{LS}).

\begin{thm}\label{lsh3}
Let $h$ be a positive $L$-harmonic function on $M$.
Then we have:
\begin{enumerate}
\item\label{lsh3a}
If $h$ is swept by $F$, then $\mu_y(h)=h(y)$ for all $y\in M$.
\item\label{lsh3b}
If $h$ is not swept by $F$, then $\mu_y(h)<h(y)$ for all $y\in M$.
\end{enumerate}
\end{thm}

\begin{proof}
\eqref{lsh3a}
Let $y\in M$.
Then $\mu_{y,0}(h)=h(y)$.
With $\tau_{y,0}=0$, assume by induction that
\begin{align}\label{lshi}
	\mu_{y,n}(h) + \sum_{1\le k\le n}\tau_{y,k}(h) = h(y)
	\quad\text{with}\quad
	\mu_{y,n}(h) \le (1-\frac1{C^2})^nh(y).	
\end{align}
By \cref{lsh1}, applied to $\mu=\mu_{y,n}$, we then get
\begin{align*}
	h(y)
	&=\mu_{y,n}(h) + \sum_{1\le k\le n}\tau_{y,k}(h) \\
	&= \mu_{y,n+1}(h) + \tau_{y,n+1}(h) + \sum_{1\le k\le n}\tau_{y,k}(h) \\
	&= \mu_{y,n+1}(h) + \sum_{1\le k\le n+1}\tau_{y,k}(h)
\end{align*}
with
\begin{align*}
	\mu_{y,n+1}(h) \le (1-\frac1{C^2})\mu_{y,n}(h) \le (1-\frac1{C^2})^{n+1}h(y).
\end{align*}
Hence \eqref{lshi} holds for all $n$.
The asserted equality $\mu_y(h)=h(y)$ is an immediate consequence.

\eqref{lsh3b}
If $h$ is not swept by $F$, there is a point $z_0\in M\setminus F$ such that $\beta(z_0,F)(h)<h(z_0)$.
Then, by the strong Markov property of the process, $\beta(z,F)(h)<h(z)$
for all $z$ in the connected neighborhood $U$ of $M\setminus F$ containing $z_0$.
Now starting at any $y\in M$, the probability of entering $U$ at some positive time is positive. 
\end{proof}

Denote by $\mathcal{H}^+(M,L)$, $\mathcal{H}^\infty(M,L)$, $\mathcal{H}^+(X,\mu)$, and $\mathcal{H}^+(X,\mu)$
the spaces of positive and bounded $L$-harmonic and $\mu$-harmonic functions on $M$ and $X$, respectively.
Let $\mathcal{H}_F^+(M,L)$ be the space of positive $L$-harmonic functions on $M$ swept by $F$.

For functions $f$ on $M$ and $g$ on $X$, denote by $Rf$ the restriction of $f$ to $X$
and by $Eg$ the function on $M$ given by $Eg(y)=\mu_y(g)$.

\begin{thm}\label{harmp}
Restriction $R\colon\mathcal{H}_F^+(M,L)\to\mathcal{H}^+(X,\mu)$ is an isomorphism with inverse $E$.
In particular, $R\colon\mathcal{H}^\infty(M,L)\to\mathcal{H}^\infty(X,\mu)$ is an isomorphism.
\end{thm}

In the case of Brownian motion, \cref{harmp} is Theorem 1.11 of \cite{BL2}.
Mutatis mutandis, the proof in \cite{BL2} carries over to our more general setting.

\subsection{$LS$-path discretization and Poisson boundary}
\label{sublsdr}
Let $(F_x,V_x)$ be $LS$-data for $X$
and $(\mu_y)_{y\in M}$ be the associated family of $LS$-measures on $X$. 
For $\omega\in\Omega$, set
\begin{equation}\label{s0}
  S_0(\omega)= 
  \begin{cases}
  0 &\text{if $\omega(0)\notin X$}, \\
  S_0^{V_x}(\omega) &\text{if $\omega(0)=x\in X$,}
\end{cases}
\end{equation}
and recursively, for $n\ge1$,
\begin{equation}\label{rs}
\begin{split}
  R_n(\omega) &= \inf \{ t\ge S_{n-1}(\omega) \mid \omega(t)\in F \}, \\
  S_n(\omega) &= \inf \{ t\ge R_n(\omega) \mid \omega(t)\notin V_{X_n(\omega)} \},
\end{split}
\end{equation}
where $X_n=X_n(\omega)\in X$ with $Y_n=Y_n(\omega)=\omega(R_n(\omega))\in F_{X_n(\omega)}$.

On $\tilde\Omega=\Omega\times[0,1]^\N$,
let $N_0(\omega,\alpha)=0$ and recursively, for $k\ge1$,
\begin{equation}\label{nk}
	N_k(\omega,\alpha)
	= \inf \{ n > N_{k-1}(\omega,\alpha) \mid \alpha_n
	< \kappa(X_n(\omega),Y_n(\omega),Z_n(\omega)) \}
\end{equation}
where we write $\alpha=(\alpha_1,\alpha_2,\dots)\in[0,1]^\N$ and $Z_n=Z_n(\omega)=\omega(S_n(\omega))$
and where
\begin{align}\label{nkk}
	\kappa = \kappa(x,y,z)
	= \frac1C\frac{d\ve(x,V_x)}{d\ve(y,V_x)}(z)
\end{align}
for $x\in X$, $y\in F_x$ and $z\in\partial V_x$.
We also set
\begin{align}\label{tk}
	T_k(\omega,\alpha)=S_{N_k(\omega,\alpha)}(\omega).
\end{align}
For $y\in M$,
let $\tilde P_y$ be the product measure $P_y\otimes\lambda^\N$ on $\tilde\Omega$,
where $\lambda$ denotes the Lebesgue measure on $[0,1]$.
Arguing as in \cite[Chapter 8]{LS}, we arrive at the following two results.

\begin{lem}\label{lsm4}
For $y\in M$, the distribution of $\tilde P_y$ at time $T_1$ is given by
\begin{align*}
	\sum_{x\in X}\mu_y(x)\ve_x^{V_x}.
\end{align*}
In particular, for all $x\in X$, we have $\tilde P_y[X_{N_1}=x]=\mu_y(x)$.
\end{lem}

\begin{proof}
View $\omega\in\Omega$ as having mass one initially and that it loses
\begin{align*}
  \kappa_n = \kappa(X_n(\omega),Y_n(\omega),Z_n(\omega))
\end{align*}
times its remaining mass to the corresponding points $x\in X$
when entering there domains $F_x$ at the hitting times $R_n$. 

By the strong Markov property of the process,
the distribution of $P_y$ at time $S_1$ is given by
\begin{align*}
	\sum_{x\in X} \int_{F_x} \beta_y^F(du) \ve_{u}^{V_{x}}
	\quad\text{and}\quad
	\sum_{x\in X} \int_{F_x} \beta_{\ve_y}^F(du) \ve_{u}^{V_{x}}
\end{align*}
for $y\in M\setminus X$ and $y\in X$, respectively,
where $\ve_y$ stands for $\ve(y,V_y)$ in the latter case.
Now for each $\omega\in\Omega$, exactly a part $\kappa(X_1(\omega),Y_1(\omega),Z_1(\omega))$
of $(\omega,\alpha)\in\tilde\Omega$ satisfy $N_1(\omega,\alpha)=1$. 
The distribution of the part $N_1=1$ of $\tilde\Omega$ at time $S_1$ is therefore
\begin{align*}
\left.
\begin{aligned}
	\frac1C\sum_{x\in X} \int_{F_x} \beta_y^F(du)\frac{d\ve_x^{V_x}}{d\ve_u^{V_x}} \ve_u^{V_x}
	&= \frac1C\sum_{x\in X} \beta_y^F(F_x) \ve_x^{V_x} \\
	\frac1C\sum_{x\in X} \int_{F_x} \beta_{\ve_y}^F(du)\frac{d\ve_x^{V_x}}{d\ve_u^{V_x}} \ve_u^{V_x}
	&= 	\frac1C\sum_{x\in X} \beta_{\ve_y}^F(F_x) \ve_x^{V_x}
\end{aligned}
	\right\} = \sum_{x\in X}\tau_{y,1}(x)\ve_x^{V_x}
\end{align*}
for $y\in M\setminus X$ and $y\in X$, respectively.
The distribution of the remaining part $N_1>1$ of $\tilde\Omega$ at time $S_1$
is then given by $\mu_{y,1}=(\mu_{y,0})'$.
Using the strong Markov property, we obtain recursively that the distribution of the parts $N_1\le k$
and $N_1>k$ of $\tilde\Omega$ at time $S_k$ are given by
\begin{align*}
	\sum_{x\in X}\sum_{1\le j\le k}\tau_{y,j}(x)\ve_x^{V_x}
	\quad\text{and}\quad
	(\mu_{y,n})',
\end{align*}
respectively.
\end{proof}

The discussion in \cite[p.\,321]{LS} applies in our setup as well (see also \cite[Theorem 2.3]{BL2} and \cite[Proposition 4]{Ka2})
and gives the following result.

\begin{thm}\label{lsmp}
The process $(X_{N_k})_{k\ge0}$ is a Markov process with time homogeneous transition probabilities $\mu_y(x)$, where $x\in X$ and $y\in M$.
More precisely,
\begin{align*}
	\tilde P_y[X_{N_1}=x_1,\dots,X_{N_k}=x_k]
	= \mu_y(x_1)\mu_{x_1}(x_2)\dots\mu_{x_{k-1}}(x_k).
\end{align*}
\end{thm}

\begin{proof}
We show that the distribution of $\tilde P_y$ restricted to $\{X_{N_1}=x_1,\dots,X_{N_k}=x_k\}$
at time $T_k$ is given by
\begin{align*}
	\mu_y(x_1)\mu_{x_1}(x_2)\dots\mu_{x_{k-1}}(x_k)\ve_{x_k}^{V_{x_k}}.
\end{align*}
For $k=1$, this is \cref{lsm4}.
Assume recursively that the assertion holds for all $j\le k-1$ for some $k\ge2$.
Then the distribution of $\tilde P_y$ of the part $\{X_{N_1}=x_1,\dots,X_{N_{k-1}}=x_{k-1}\}$
of $\tilde\Omega$ at time $T_{n-1}$ is given by
\begin{align*}
	\sum \mu_y(x_1)\mu_{x_1}(x_2)\dots\mu_{x_{k-2}}(x_{k-1})\ve_{x_{k-1}}^{V_{x_{k-1}}}.	 
\end{align*}
Now $x_{k-1}\in X$, and hence the second alternative in the definition of the LS-measures applies.
That is, $(X_{N_1})$ with initial distribution $\ve(x_{k-1},V_{x_{k-1}})$ has distribution at time $T_1$
as given in \cref{lsm4} (with $y$ substituted by $x_{k-1}$).
Now using the strong Markov property, we arrive at the asserted equality.
\end{proof}

\begin{thm}\label{lspo}
$LS$-path discretization induces an isomorphism
\begin{align*}
	\mathcal{P}(X,\mu)\to\mathcal{P}(M,L)
\end{align*}
of Poisson boundaries.
\end{thm}

We follow the arguments in the proof of the corresponding \cite[Theorem 1]{Ka2}.

\begin{proof}[Proof of \cref{lspo}]
As in \cite[Theorem 1]{Ka2},
we decompose $LS$-path discretization into the following four measure preserving maps between Markov processes,
\begin{align*}
	\Omega\ni\omega &\xrightarrow[1]{}
	(\omega(R_n(\omega)),\omega(S_n(\omega)))  \xleftarrow[2]{}
	(\omega(R_n(\omega)),\omega(S_n(\omega)),\alpha_n)  \xrightarrow[3]{} \\
	&\xrightarrow[3]{}
	(\omega(R_{N_k}(\omega)),\omega(S_{N_k}(\omega)),\alpha_{N_k})  \xrightarrow[4]{}
	(X_{N_k(\omega)}(\omega)) \in M\times X^\N.
\end{align*}
Since the $\alpha_n$ are i.i.d. and independent of $\omega$,
the second map induces an isomorphism of the corresponding Poisson boundaries.
Thus we obtain a measurable map $\mathcal{P}(M,L)\to\mathcal{P}(X,\mu)$. 
By \cref{harmp},
the spaces of bounded harmonic functions on $M$ and $X$ are isomorphic to each other under this map.
Hence it is an isomorphism of Poisson boundaries.
\end{proof}

The difference to the argument in \cite{Ka2} is that we can use the isomorphism $\mathcal{H}^\infty(M,L)\to\mathcal{H}^\infty(X,\mu)$ from \cref{harmp},
whereas in \cite{Ka2}, the aim of the proof is to establish the latter.

\subsection{Green functions and Martin boundary}
\label{submar}

For $x\in X$ and $y\in M$, we let
\begin{equation}\label{green1}
  g(y,x) = \delta_y(x) + \sum_{k\ge1} \tilde P_y(X_{N_k}=x),
\end{equation}
the \emph{Green function} associated to the Markov chain $(X_{N_k})_{k\ge1}$.
Recall that $g(x,y)<\infty$ for some, or any, pair $x\ne y$ of points in $X$
if and only if the $\mu$-random walk on $X$ is transient.

\begin{prop}\label{green2}
For all $x\in X$ and $y\in M\setminus V_x$, we have
\begin{equation*}
  g(y,x) = \frac{1}{C} \sum_{n=1}^{\infty} \nu_{y,n}(F_{x}),
\end{equation*}
where $\nu_{y,n}$ denotes the distribution of $P_y$ w.r.t. $\omega(R_{n}(\omega))$, that is,
\begin{equation*}
  \nu_{y,n}(A) = P_y[\omega(R_n(\omega))\in A].
\end{equation*}
\end{prop}

In the case of Brownian motion, \cref{green2} is Equation 2.5 of \cite{BL2}.
Mutatis mutandis, the proof in \cite{BL2} carries over to our setting.

The \emph{Green function} $G$ of the diffusion process $(D_t)$ is given by
\begin{equation}\label{Green1}
	G(x,y) = \int_{0}^{\infty} p(t,x,y)\,dt.
\end{equation}
Keeping in mind that $p(t,x,y)$ is the transition density with respect to the volume element $\vf^2\dv$,
the Green function measures the average time of sojourn of a path in subsets of $M$ in the sense that,
for any Borel subset $B \subseteq M$, we have
\begin{equation}\label{Green2}
    \int_{B} G(x,y) \vf^2(y)\dv(y) = E_{x}\left(\int_{0}^{\infty} \chi_{B}(\omega(t))\,dt\right)
\end{equation}
Recall that $G(x,y)<\infty$ for some, or any, pair $x\ne y$ of points in $M$
if and only if the diffusion process $(D_t)$ on $M$ associated to $L$ is transient.
Then $G$ is smooth on the set of pairs $(x,y)$ with $x\ne y$.

More generally, for any open subset $V$ of $M$, we let $G_V$ be the Green function of $V$.
It is obtained by stopping the diffusion process $(D_t)$ upon leaving $V$.

\begin{prop}\label{Green4}
For all $x\in X$ and $y\in M\setminus V_x$, we have
\begin{equation*}
  G(y,x) = \sum_{n = 1}^{\infty} \int_{\partial F_{x}} G_{V_x}(z,x)\nu_{y,n}(dz).
\end{equation*}
\end{prop}

In the case of Brownian motion, \cref{green2} is Equation 2.6 of \cite{BL2}.
Mutatis mutandis, the proof in \cite{BL2} carries over to our setting.

Following \cite{BL2}, we say that $LS$-data $(F_x,V_x)$ for $X$ are \emph{balanced}
if there is a constant $B$ such that
\begin{enumerate}
\item[(D5)]\label{balan}
$G_{V_x}(z,x) = B$ for all $x\in X$ and $z\in\partial F_x$.
\end{enumerate}
The following result extends \cite[Theorem 2.7]{BL2}.

\begin{thm}\label{Greengreen}
If $(F_x,V_x)$ are balanced $LS$-data for $X$ with constants $B$ and $C$, then
\begin{align*}
	G(y,x) = BCg(y,x)
	\quad\text{for all $x\in X$ and $y\in M\setminus V_x$}.
\end{align*}
In particular, the diffusion process $(D_t)$ is transient if and only if the $\mu$-random walk on $X$ is transient.
In this case, $\mu_{x}(y) = \mu_{y}(x)$ for all $x,y \in X$.
\end{thm}

\begin{proof}
The displayed formula follows immediately from \cref{green2} and \cref{Green4}.
For the assertion about the transience, we recall that $(D_t)$ is transient if and only if $G(x,y)<\infty$ for some--and then all--points $x\ne y$ in $M$
and that the $\mu$-random walk on $X$ is transient if and only if $g(x,y)<\infty$ for some--and then all--points $x,y$ in $X$.

Since $G(x,y) = G(y,x)$ for all $x,y\in X$, we conclude that $g(x,y)=g(y,x)$ for all $x,y\in X$.
The proof of the symmetry of $\mu$ is now exactly the same as in the case of Brownian motion in \cite[Theorem 2.7]{BL2}.
\end{proof}

Throughout the rest of this subsection, we assume that $M$ is $L$-transient, $X$ is $*$-recurrent, and $X$ is endowed with the family $\mu$ of LS-measures associated to balanced LS-data $(F_x,V_x)$.
Given a subset $S\subseteq M$, we denote by $\bar{S}$ its closure in the Martin compactification $\mathcal{M}(M,L)$ and set
\[
\partial_L S = \bar{S} \cap \partial_L M.
\]

\begin{lem}\label{muhar}
	\begin{enumerate}
		\item\label{heh}
		For any $h\in\mathcal H^+(X,\mu)$, $Eh$ is the unique positive $L$-harmonic function on $M$ extending $h$.
		\item\label{cont}
		The extension map $E \colon \mathcal{H}^+(X,\mu) \to \mathcal{H}^+_F(M,L)$ is continuous with respect to the compact-open topology.
		\item\label{xbx}
		For any $\xi\in\partial_\mu X\cap\mathcal H^+(X,\mu)$,
		there is a unique $E\xi\in\partial_L M$ such that $Ek(.,\xi)=K(.,E\xi)$.
		Moreover, $E\xi\in\partial_L X\cap\mathcal H^+_F(M,L)$.
	\end{enumerate}
\end{lem}

\begin{proof}
	\eqref{heh}
	If $f\in\mathcal H^+(M,L)$ satisfies $Rf=h$, then $R(Eh-f)=0$.
	Furthermore, by Theorem \ref{lsh3}, we have
	\begin{align*}
		Eh = ERf \le f
	\end{align*}
	and hence $f-Eh$ is a non-negative $L$-harmonic function vanishing on $X$.
	Therefore $f-Eh=0$ by the mean value property of $L$-harmonic functions.
	
	\eqref{cont} Let $(f_n)_{n \in \mathbb{N}}$ be a sequence in  $\mathcal{H}^+(X,\mu)$ converging to $f\in\mathcal{H}^+(X,\mu)$  pointwise.
	Consider the sequence of extensions $(Ef_n)_{n \in \mathbb{N}}$ in $\mathcal{H}^+_F(M,L)$ and choose $x_0 \in X$.
	Given any subsequence $(Ef_{n_k})_{k \in \mathbb{N}}$, since it is bounded at $x_0$, it follows from the gradient estimate that $(Ef_{n_k})_{k \in \mathbb{N}}$ is uniformly bounded and equicontinuous on any compact domain of $M$ containing $x_0$.
	We derive from the Arzela-Ascoli theorem that, after passing to a subsequence if necessary, we have that $Ef_{n_k} \rightarrow g \in C(M,\mathbb{R})$ locally uniformly.
	Using again that $E f_{n_k}$ is $L$-harmonic, we deduce that $g \in \mathcal{H}^+(M,L)$.
	Moreover, since $E f_{n_k} \rightarrow g$ locally uniformly,
	we readily see that $f_{n_k} = R E f_{n_k} \rightarrow Rg$ pointwise. This, together with the fact that $f_n \rightarrow f$ pointwise, gives that $f = Rg$, and hence, $g = Ef$. This means that any subsequence of $(E f_n)_{n \in \mathbb{N}}$ has a subsequence converging locally uniformly to $Ef$, which yields that $Ef_n \rightarrow Ef$ in $H^+_F(M,L)$. Since the space $H^+(X,\mu)$ is metrizable, this establishes the continuity of the extension map.
	
	\eqref{xbx}
	Let $(x_n)$ be a sequence in $X$ that converges to $\xi$ in $\mathcal M_\mu X$.
	Since $\mathcal M_L M$ is compact, $K(.,x_n)$ subconverges in $M$ to some $K(.,\eta)$.
	Since $k(.,x_n)=RK(.,x_n)$ on $X \setminus \{x_n\}$ for all $n$, we conclude that $k(.,\xi)=RK(.,\eta)$.
	Since $k(.,\xi)$ is $\mu$-harmonic and $K(.,\eta)$ is $L$-harmonic, the first part of the proof implies the assertions.
\end{proof}

\begin{thm}\label{bp330}
	Restriction $R\colon\partial_L X\to\partial_\mu X$ of Martin kernels from $M$ to $X$
	is continuous, closed, and surjective.
	Furthermore,
	\begin{align*}
		R\colon\partial_L X\cap\mathcal H^+_F(M,L) \to \partial_\mu X\cap\mathcal H^+(X,\mu)
	\end{align*}
	is a homeomorphism with inverse $E$.
\end{thm}

\begin{proof}
	Let $\xi\in\partial_L X$.
	Then there is a sequence $(x_n)$ in $X$ converging to $\xi$ in $\mathcal M_L M$.
	By definition, for any such sequence,
	the Martin kernels $K(.,x_n)$ converge pointwise to the Martin kernel $K(.,\xi)$ corresponding to $\xi$.
	Hence their restrictions $k(.,x_n)$ to $X$ converge pointwise as functions on $X$.
	By definition, this means that the sequence of $x_n$ converges to a point $\eta$ in the Martin boundary $\partial_\mu X$ of $X$,
	where the associated Martin kernel $k(.,\eta)$ is the restriction of $K(.,\xi)$ to $X$. This means that $R \colon \partial_L X \to \partial_\mu X$ is well defined. Continuity of $R$ is obvious, since limits correspond to pointwise convergence of Martin kernels. This yields that $R$ is closed, $\partial_L X$ being compact.

	Let $(x_n)$ be a sequence in $X$ which converges to a point $\xi$ in $\partial_\mu X$.
	Now $\mathcal M_L M$ is compact.
	Hence the sequence of $x_n$ subconverges to a point $\eta\in\partial_L M$.
	It follows that the restriction of the associated Martin kernel $K(.,\eta)$ to $X$ equals $k(.,\xi)$.
	Thus $R$ defines a surjection from $\partial_L X$ onto $\partial_\mu X$.
	
	By \cref{muhar}(\ref{xbx}), $R\colon\partial_L X\cap\mathcal H^+_F(M,L)\to\partial_\mu X\cap\mathcal H^+(X,\mu)$ is bijective.
	Since $R$ is continuous and closed, it is a homeomorphism.
\end{proof}

The spaces $\partial_{L}^{\min}M$ of minimal positive $L$-harmonic functions on $M$
and $\partial_{L,F}^{\min}M$ of them swept by $F$ are Borel subsets of $\partial_LM$.
For $h\in\mathcal H^+(M,L)$, recall the representation formula \eqref{poisson2}.
Analogously, the space $\partial_\mu^{\min}X$ of minimal positive $\mu$-harmonic functions on $X$ is a Borel subset of $\partial_\mu X$,
and there is a representation of positive $\mu$-harmonic functions on $X$ analogous to \eqref{poisson2}.

\begin{cor}\label{bl29}
	For any $h\in\mathcal H^+(M,L)$ and $\nu_h$ according to \eqref{poisson2},
	\begin{align*}
		h\in\mathcal H^+_F(M,L) \quad\text{if and only if}\quad \nu_h(\partial_{L,F}^{\min}M)=\nu_h(\partial_L^{\min}M).
	\end{align*}
\end{cor}

\begin{proof}
	We have $ERh\le h$ with equality if and only of $h\in\mathcal H^+_F(M,L)$.
	Hence $ERh=h$ if and only if $ERK(.,\xi)=K(.,\xi)$ for $\nu_h$-almost all $\xi$.
\end{proof}

\begin{thm}\label{mfx}
We have $\partial_{L,F}^{\min}M\subseteq\partial_LX\cap\mathcal H^+_F(M,L)$,
and restriction \[R\colon \partial_{L,F}^{\min}M \to \partial_\mu^{\min}X\] is a homeomorphism with inverse $E$.
\end{thm}

\begin{proof}
We start with the proof of the second assertion.
To that end, let $\xi\in\partial_{L,F}^{\min}M$ and $h=K(.,\xi)$.
Suppose that $Rh\ge f$ for some $f\in\mathcal H^+(X,\mu)$.
Then $h = ERh \ge Ef$, and hence $Ef$ is a multiple of $h$.
But then also $f=REf$ of $Rh$. This means that $R\xi \in \partial_\mu^{\min}X$. Injectivity of $R$ follows readily from Lemma \ref{muhar}(\ref{heh}). 	

To establish the surjectivity of $R$, let now $\xi \in \partial_\mu^{\min}X$ and $h=k(.,\xi)$.
We know from Theorem \ref{bp330} that $E\xi \in \partial_L M \cap\mathcal H^+_F(M,L)$,
and it suffices to show that $Eh$ is minimal. Assume to the contrary that this is not the case. Then there exist distinct $h_1, h_2 \in \mathcal{H}^+(M,L)$ and $0 < c < 1$ such that $E h = c h_1 + (1-c)h_2$. Using that $Eh$ is swept by $F$, it follows from Theorem \ref{lsh3} that $h_1$ and $h_2$ are also swept by $F$. In particular, this yields that $Rh_1 , Rh_2 \in \mathcal{H}^+(X,\mu)$ and $Rh_1 \neq Rh_2$, in view of Lemma \ref{muhar}(\ref{heh}). Since $h = c R h_1 + (1-c) Rh_2$, we conclude that $h$ is not minimal, which is a contradiction.

It remains to prove the first assertion.
Again, let $\xi\in\partial_{L,F}^{\min}M$ and $h=K(.,\xi)$.
Then $R\xi\in\partial_\mu^{\min}X$ by the first part of the proof.
Then $ER\xi\in\partial_L X\cap\mathcal H^+_F(M,L)$ by \cref{bp330}.
Since $h=ERh$ on $X$ and $h$ is swept by $F$, we have $\xi=ER\xi$, and the assertion follows.
\end{proof}

\begin{cor}\label{mfx2}
We have that
\begin{align*}
    \partial_{L,F}^{\min}M = \partial_L X\cap\mathcal H^+_F(M,L) \quad\text{if and only if}\quad \partial_\mu^{\min}X = \partial_\mu X\cap\mathcal H^+(X,\mu).
\end{align*}
\end{cor}

\begin{proof}
The left hand sides of the equations are contained in the right hand sides, in the first case by \cref{mfx}, in the second by \cite[Theorem 7.53]{Wo}.
Now the assertion follows from Theorems \ref{bp330} and \ref{mfx} since the homeomorphisms there are both given by $R$. 
\end{proof}

\cref{mfx} and \cref{mfx2} establish an analogue of \cref{mami} in the setting of this subsection.

\subsection{The Martin boundary in the uniform case}
Let $X$ be a $*$-recurrent subset of $M$ and $(F_x,V_x)$ balanced LS-data for $X$. We say that $(F_x,V_x)$ are \emph{uniform} LS-data if there exist families $(U_x)_{x \in X}$ of Borel subsets and $(W_x)_{x \in X}$ of relatively compact open subsets of $M$ satisfying the following:
 \begin{enumerate}[label=(U\arabic*)]
 	\item\label{dc1} $F_x \subseteq U_x \subseteq W_x$ and $F_x \cap W_y = \emptyset$ for any $x,y \in X$ with $x \neq y$,
 	\item\label{dc2} $(U_x)_{x \in X}$ is a partition of $M$ and $(W_x)_{x \in X}$ is locally finite,
 	\item\label{dc3} for all $x \in X$, $y,z \in U_x$, and positive harmonic function $h$ on $W_x$,
 	\[
 	h(y) \geq c h(z),
 	\]
 	\item\label{dc4} for any $x\in X$ and $y\in U_x$, the balayage $\beta_y$ of the Dirac measure $\delta_y$ onto the closed set $F_x\cup(M\setminus W_x)$, satisfies
 	\begin{align*}
 		\beta_y(F_x) \ge c.
 	\end{align*}
 \end{enumerate}
Here, the constants $0<c<1$ in \ref{dc3} and \ref{dc4}, if they exist, are chosen to coincide. If $X$ admits uniform LS-data, then $X$ is called \emph{$*$-uniform}.

\begin{thm}\label{aim}
	If $M$ is $L$-transient and $X$ is $*$-uniform,
	then
	\begin{align*}
		\mathcal H^+(M,L)=\mathcal H^+_F(M,L).
	\end{align*}
\end{thm}

\begin{proof}
	Choose uniform $LS$-data $(F_x,V_x)$ and families $(U_x,W_x)$ as above, and let $h\in\mathcal H^+(M,L)$.
	Because $h$ is bounded harmonic on the relatively compact sets $W_x$, we have
	\begin{align*}
		h(y) = \beta_y(h) = \beta_y(h|_{F_x}) + \beta_y(h|_{\partial 
			W_x})
	\end{align*}
	for all $x\in X$ and $y\in U_x$.
	By the above assumptions, we conclude that
	\begin{align}\label{estf}
		\beta_y(h|_{F_x}) \ge c^2 h(y)
	\end{align}
	for any $y\in U_x$.
	
	For any measure $\mu$ on $U_x$, let $\beta_{\mu}$ be the balayage of $\mu$ onto the closed set $F_x\cup(M\setminus W_x)$,
	\begin{align*}
		\beta_{\mu} = \int\beta_{y}\mu(dy).
	\end{align*}
	More generally, if the support of $\mu$ is not contained in $U_x$, set
	\begin{align*}
		\beta_\mu = \sum_{x\in X}\beta_{\mu_x} = \sum_{x\in X}\int_{U_x}\beta_{y}\mu(dy),
	\end{align*}
	where $\mu_x$ denotes the restriction of $\mu$ to $U_x$.
	Now given any finite measure $\mu_n$ on $M$, set
	\begin{align*}
		\mu_{n+1} = \sum_x\beta_{\mu_{nx}}|_{\partial W_x}
		\quad\text{and}\quad
		\tau_{n+1} = \sum_x\beta_{\mu_{nx}}|_{F_x} = \beta_{\mu_n}|_F.
	\end{align*}
	Recall here that the $F_x$ and $\partial W_y$ do not intersect,
	and note that the $\partial W_x$ are not all disjoint since the $U_x$ form a partition of $M$.
	However, the $\partial W_x$ are a locally finite family of subsets of $M$.
	
	For any finite Borel measure $\mu_n$, we have
	\begin{align*}
		\mu_{n+1}(h) + \tau_{n+1}(h) = \mu_n(h).
	\end{align*}
	Moreover, by \eqref{estf},
	\begin{align*}
		\tau_{n+1}(h) \ge c^2\mu_n(h).
	\end{align*}
	Therefore, if $\mu_n(h)<\infty$, then
	\begin{align*}
		\mu_{n+1}(h) \le (1-c^2)\mu_n(h).
	\end{align*}
	Now to show the assertion, let $y\in M$.
	If $y\in F$, we are done.
	If $y\in M\setminus F$, let $\mu_0=\delta_y$ and $\tau_0=0$.
	Then $\mu_0(h)+\tau_0(h)=h(y)<\infty$.
	Recursively by the above procedure, we obtain measures $\mu_n$ and $\tau_n$ such that
	\begin{align*}
		\mu_{n}(h) + \sum_{k\le n}\tau_{k}(h) = h(y)
		\quad\text{and}\quad
		\mu_n(h) \le (1-c^2)^nh(y).
	\end{align*}
	Now $\sum_n\tau_n=\beta_F$, the balayage of $\delta_y$ onto $F$ and therefore $\beta_F(h) = h(y)$.
	This is the assertion.
\end{proof}

\begin{cor}\label{unifbound}
If $M$ is $L$-transient, $X$ is $*$-uniform, and $X$ is endowed with the family $\mu$ of $LS$-measures associated to uniform $LS$-data $(F_x,V_x)$, then
\[
\partial_\mu X = \partial_\mu X \cap H^+(X,\mu).
\]
\end{cor}

\begin{proof}
Follows readily from Theorems \ref{bp330} and \ref{aim}.
\end{proof}

\begin{thm}\label{martin}
Assume that $M$ is $L$-transient, $X$ is $*$-uniform,
and $X$ is endowed with the family $\mu$ of $LS$-measures associated to uniform $LS$-data $(F_x,V_x)$.
Then 
\begin{enumerate}
\item\label{marmin}
the $\mu$-random walk on $X$ is transient
and the inclusion $X\to M$ extends to a homeomorphism
\begin{align*}
	\partial_\mu^{\min}X\to\partial_L^{\min}M; 
\end{align*}
\item\label{parmin}
$\partial_L^{\min}M=\partial_LM$ if and only if $\partial_\mu^{\min}X=\partial_\mu X$.
\end{enumerate}
\end{thm}

\begin{proof}
Choose the origin $x_0$ for the Martin kernels associated to $L$ and $\mu$ in $X$; see \eqref{marker}.
Then they coincide on all points $(y,x)$ in $M\times X$ with $y\notin V_x$ and $x\ne x_0$,
by \cref{Greengreen}. Let $(U_x,W_x)$ be families as in the definition of uniform LS-data.

Let $\eta\in\partial_L^{\min}M$
and  denote the associated minimal positive $L$-harmonic function on $M$ by $K(.,\eta)$. It follows from \cref{aim} that $K(.,\eta) \in \mathcal{H}_F^+(M,L)$, which means that $\eta \in \partial_{L,F}^{\min} M$. In view of \cref{mfx}, we derive that $\eta \in \partial_L X$ and $R \eta \in \partial_{\mu}^{\min}X$. In this sense, $\partial_L^{\min}M\subseteq\partial_\mu^{\min}X$.
Conversely, the extension of a minimal positive $\mu$-harmonic function on $X$ is a minimal positive $L$-harmonic function on $M$, by virtue of Theorem \ref{mfx}.

As for \eqref{parmin},
suppose that $\partial_\mu^{\min}X=\partial_\mu X$.
Consider $\eta \in \partial_L M$ and a diverging sequence of points $y_n\in M$ such that $K(.,y_n)\to K(.,\eta)$.
Choose $x_n\in X$ with $y_n\in U_{x_n}$.
Then the sequence of $x_n$ is diverging and, up to passing to a subsequence, we have that $k(.,x_n) \to k(.,\xi)$ and $K(.,x_n) \rightarrow K(.,E\xi)$ for some $\xi \in \partial_\mu X = \partial_\mu^{\min}X$. From the above, we deduce that $K(.,E\xi)$ is minimal. Keeping in mind that the $G(.,x)$ are $L$-harmonic on $M\setminus\{x\}$, the uniform Harnack constants of \ref{dc3} imply that, for any $x\in M$,
\begin{align*}
  K(x,y_n) < c^{-2} K(x,x_n)
\end{align*}
for all sufficiently large $n$, which yields that  $K(.,\eta)\le c^{-2} K(.,E \xi)$.
Since $K(.,E \xi)$ is minimal and $K(x_0, \eta)=K(x_0,E \xi)=1$,
we conclude that $K(.,\eta)=K(.,E\xi)$. This means that $\eta = E \xi \in \partial_L^{\min} M$. The proof of the other direction is similar, in view of Theorem \ref{aim}.
\end{proof}

Since any positive $L$-harmonic function on $M$ and positive $\mu$-harmonic function on $X$
can be written as an average of minimal positive $L$-harmonic and $\mu$-harmonic functions,
respectively, \cref{martin} has the following consequence.

\begin{cor}\label{harmc}
Assume that $M$ is $L$-transient, that $X$ is $*$-uniform,
and that $X$ is endowed with the family $\mu$ of $LS$-measures associated to uniform $LS$-data $(F_x,V_x)$.
Then
\begin{enumerate}
\item\label{harmps}
for any $h\in\mathcal{H}^+(M,L)$, the restriction $h|_X$ of $h$ to $X$ belongs to $\mathcal{H}^+(X,\mu)$.
More precisely, $h(y)=\mu_y(h)$ for all $y\in M$;
\item\label{harmpl}
the restriction map
\begin{align*}
	\mathcal{H}^+(M,L)\to\mathcal{H}^+(X,\mu), \quad h\mapsto h|_X,
\end{align*}
is an isomorphism of cones.
\end{enumerate}
\end{cor}

\section{Diffusion processes and covering projections}
\label{secdico}

Let $\pi\colon M\to N$ be a covering of connected manifolds
with group $\Gamma$ of covering transformations.
We do not assume that $\pi$ is normal, that is,
that the action of $\Gamma$ is transitive on the fibers of $\pi$.
Let $L_0$ be an elliptic diffusion operator on $N$ and $L$ be the pull back of $L_0$ to $M$.
Let $(E_t)_{t\ge0}$ be the diffusion process on $N$ associated to $L_0$.
For simplicity we assume here that $(E_t)$ has infinite life time.
Then we may view $(E_t)$ as the evaluation at time $t\ge0$
on the space $\Omega_N$ of continuous paths $\omega\colon\R_{\ge0}\to N$
together with a family of probability measures $(Q_y)_{y\in N}$ on $\Omega_N$
which are determined by their distributions $Q_y[\omega(t)\in A]$, $A\in\mathcal{B}(N)$, at times $t\ge0$,
where $\mathcal{B}(N)$ denotes the $\sigma$-algebra of Borel sets in $N$.

By the path lifting property of $\pi$, the map
\begin{align}\label{omega}
	H \colon \{(x,\omega)\in M\times\Omega_N\mid\pi(x)=\omega(0)\} \to \Omega_M, \quad
	H(x,\omega) = \omega_x,	
\end{align}
where $\omega_x$ denotes the continuous lift of $\omega$ to $M$ starting at $x$,
is a homeomorphism with respect to the compact-open topologies.
In what follows, we identify $\Omega_M$ accordingly.
Evaluation $D_t$ at time $t\ge0$ on $\Omega_M$ is then given by $D_t(x,\omega)=\omega_x(t)$.

\begin{thm}\label{liftp}
For $x\in M$ with $\pi(x)=y$, define the probability measure $P_x$ on $\Omega_M$ by 
\begin{align*}
	P_x[A] = Q_y[\{\omega\mid(x,\omega)\in A\}], \quad A\in\mathcal{B}(\Omega_M).
\end{align*}
Then $(D_t)_{t\ge0}$ together with the probability measures $(P_x)_{x\in M}$
is the diffusion process on $M$ with generator $L$ on $C^\infty_c(M)$.
\end{thm}

\begin{proof}
Let $V\subseteq N$ be a relatively compact and connected open domain with smooth boundary
such that a neighborhood $V'$ of $\bar V$  is evenly covered by $\pi$,
and let $U'$ be a sheet of $\pi$ over $V'$ .
Then $\pi\colon U' \to V'$ is a diffeomorphism with $\pi^*L_0=L$.
Hence the diffusion processes associated to $L$ on $U=\pi^{-1}(V)\cap U'$ and $L_0$ on $V$, started at any $x\in U$ respectively $y=\pi(x)\in V$,
correspond to each other until exiting $U$ respectively $V$.
By the strong Markov property of diffusion processes,
this implies that $(D_t)$ is the diffusion associated to $L$.
\end{proof}

\begin{cor}\label{liftpc}
For all $x\in M$ with $\pi(x)=y$, we have
\begin{align*}
	 Q_y[\omega(t)\in A]
	 = P_x[\omega(t)\in\pi^{-1}(A)], 
	 \quad A\in\mathcal{B}(N).
\end{align*}
In particular, if the distributions of $D$ and $E$ are given by densities $p=p(t,x,u)$ on $M$ and $q=q(t,y,v)$ on $N$
with respect to associated smooth volume elements on $M$ and $N$, then
\begin{equation*}
	q(t,y,v) = \sum_{u\in\pi^{-1}(v)} p(t,x,u).
\end{equation*}
\end{cor}

From now on, we assume that $N$ is $L_0$-recurrent, that is,
that the diffusion process on $N$ associated to $L_0$ is recurrent.
We choose $z_0\in N$ and let $X=\pi^{-1}(z_0)$, a discrete subset of $M$.
\begin{enumerate}[label=(C)]
\item\label{fv}
We let $F_0\subseteq V_0$ be connected neighborhoods of $z_0$ in $N$
such that $F_0$ is compact, $V_0$ is open and relatively compact,
and $V_0$ is evenly covered by $\pi$.
For $x\in X$,
we let $F_x\subseteq V_x$ be the connected components of $F=\pi^{-1}(F_0)$ and $V=\pi^{-1}(V_0)$ containing $x$.
\end{enumerate}
Any such family $(F_x,V_x)_{x\in X}$, is \emph{$\Gamma$-equivariant}, that is,
$\gamma F_x=F_{\gamma x}$ and $\gamma V_x=V_{\gamma x}$ for all $\gamma$ in the group $\Gamma$ of covering transformations of $\pi$.

\begin{lem}\label{lsc}
If $N$ is $L_0$-recurrent, then
\begin{enumerate}
\item\label{lsca}
any family $(F_x,V_x)$ satisfying \ref{fv} constitutes $LS$-data for $X$,
and the associated $LS$-measures are $\Gamma$-equivariant.
\item\label{lscb}
$X$ admits balanced $LS$-data satisfying \ref{fv}.
\end{enumerate}
Moreover, if $N$ is compact, then $X$ is $*$-uniform and, in particular, admits uniform LS-data satisfying \ref{fv}.
\end{lem}

\begin{proof}
\ref{lsca})
Since $N$ is connected and $L_0$-recurrent,
the diffusion process on $N$ associated to $L_0$ started at any point of $N$ will hit $F_0$ with probability one.
Hence its lift started at any point of $M$ will hit the preimage $F$ of $F_0$ with probability one.
Therefore $F$ is recurrent.
The existence of a uniform Harnack constant is clear.

\ref{lscb})
Let $V_0$ be a connected and relatively compact open neighborhood of $z_0$ in $N$ which is evenly covered by $\pi$.
Let $G_{V_0}$ be the Green function of $V_0$ for the operator $L_0$.
Then $G_{V_0}(z_0,.)$ vanishes on the boundary of $V_0$ and $G_{V_0}(z_0,z)$ tends to infinity as $z\in V_0$ tends to $z_0$.
Moreover, under $\pi$, $G_{V_0}$ corresponds to the Green functions $G_{V_x}$ of the $V_x$.
Since $V_0$ is evenly covered, we have $V_x\cap V_y=\emptyset$ for all $x\ne y$ in $X$.
Hence for any constant $B>0$,
the connected components $F_x$ of $\{G_{V_x}\ge B\}$ containing $x$ together with the $V_x$ are balanced $LS$-data that satisfy \ref{fv}.

It remains to prove the last assertion.
Now the Riemannian metric on $N$ associated to $L_0$ is complete since $N$ is compact.
Its pull back to $M$ is complete and equals the Riemannian metric associated to $L$.
Choose $R > 0$ less than half the injectivity radius of $N$ at $z_0$.
Then the open ball $B(z_0,2R)$ is evenly covered.
Consider $F_0 \subseteq V_0 \subseteq B(z_0,R)$ such that the associated LS-data $(F_x,V_x)$ for $X$ as in (\ref{lscb}) are balanced.
For $x\in X$, consider the Dirichlet domain
\[
D_x = \{ z \in M : d(z,x) \leq d(z,y) \text{ for all } y \in X \}
\]
centered at $x$. Since $(D_x)_{x \in X}$ is a closed cover of $M$, by removing closed subsets from them,
we obtain a partition $(U_x)_{x \in X}$ of $M$ (not necessarily $\Gamma$-equivariant), consisting of Borel subsets.
Finally, let $W_x = B(D_x,R)$ be the open $R$-tubular neighborhood of $D_x$ with $x \in X$.

It is easily checked that $B(x,R) \cap W_y = \emptyset$ for all $x,y \in X$ with $x \neq y$, which implies that $F_x \subseteq B(x,R) \subseteq U_x$ and $F_x \cap W_y = \emptyset$ for all $x,y \in X$ with $x \neq y$. Keeping in mind that $\diam D_x \leq \diam M$ for any $x \in X$, it is easily checked that if $z \in M$ belongs to the intersection of $n$ different $W_x$'s, then there exist $n$ pairwise non-homotopic loops based at $z_0$ of length at most $2 \diam M + 2R$. This yields that $(W_x)_{x \in X}$ is locally finite. Property \ref{dc3} follows from the gradient estimate of Cheng and Yau \cite[Theorem 6]{CY}, keeping in mind that the curvature of $M$ is bounded and that $D_x$ is star-shaped with respect to $x$ and is contained in the closed ball $C(x,\diam M)$, for any $x \in X$.

It remains to show that \ref{dc4} is satisfied. To this end, let $r > 0$ such that $B(z_0,2r) \subseteq F_0$. It is clear that there exists $c > 0 $ such that
\[
\varepsilon_y^{V_0 \setminus C(z_0,r)}(C(z_0,r)) \geq c
\]
for any $y \in F_0 \setminus C(z_0,r)$. For $x \in X$, consider the positive harmonic function $h_x$ on $W_x \setminus B(x,r)$ defined by
\[
h_x(y) = \varepsilon_y^{W_x \setminus C(x,r)}(C(x,r)).
\]
Then $h_x(y) \geq \varepsilon_y^{V_x \setminus C(x,r)}(C(x,r)) \geq c$ for any $x \in X$ and $y \in F_x \setminus C(x,r)$. It is evident that for any point in $D_x \setminus F_x$ there exists a point on $\partial F_x$ and a geodesic joining them, of length at most $\diam M$ and image contained in $D_x \setminus F_x^\circ$ (for instance, a segment of a geodesic emanating from $x$).
Since $h_x$ is harmonic in the $r$-tubular neighborhood of $D_x \setminus F_x^\circ$ for any $x \in X$ and the curvature of $M$ is bounded, it follows from \cite[Theorem 6]{CY} that there exists $C>0$ such that
\[
h_x(y) \geq C  \min_{z \in \partial F_x } h_x(z) \geq C c
\]
for any $x \in X$ and $y \in D_x \setminus F_x$. From the definition of $h_x$, it is evident that for any $x \in X$ and $y \in U_x \setminus F_x \subseteq D_x \setminus F_x$, the balayage $\beta_y$ of $\delta_y$ onto $F_x \cup (M \setminus W_x)$ satisfies $\beta_y(F_x) \geq h_x(y) \geq Cc$.
\end{proof}

\begin{proof}[Proof of Theorems \ref{poha} - \ref{trec} in the introduction]
\cref{poha} and \cref{mami} are immediate consequences of \cref{martin} and \cref{harmc},
using that $X$ is $*$-uniform and choosing uniform $LS$-data satisfying \ref{fv} (using \cref{lsc}.\ref{lscb}).

For the first three assertions of \cref{boha}, we need $LS$-data satisfying \ref{fv}, not necessarily balanced.
With such data,
Assertion \ref{bohar} is a consequence of \cref{harmp},
using that bounded harmonic functions are swept by $F$ (\cref{swel}).
Assertion \ref{bohap} follows from \cref{lsh3} and Assertion \ref{bohab} from the second assertion in \cref{harmp}.
Assertion \ref{bohat} follows from \cref{Greengreen},
where we now choose balanced $LS$-data satisfying \ref{fv} (using \cref{lsc}.\ref{lscb}).

Assertion \ref{recp} of \cref{trec} follows from \cref{lspo} with any choice of $LS$-data satisfying \ref{fv}.
Assertion \ref{recm} is a consequence of \cref{bp330},
where we choose balanced $LS$-data satisfying \ref{fv} (using \cref{lsc}.\ref{lscb}).
\end{proof}

\section{Diffusion processes and properly discontinuous actions}
\label{secdiac}
In the case of a covering $\pi\colon M\to N$ considered in \cref{secdico},
the group $\Gamma$ of covering transformations acts properly discontinuously
and freely on $M$
and the diffusion operator $L$ on $M$ is the pull back of a diffusion operator on $N$.
The action of $\Gamma$ on the fibers of $\pi$ is transitive if the covering is normal.
We consider now the case of a properly discontinuous action of a group $\Gamma$ on $M$,
extending the case of normal coverings.
As before, we let $L$ be a diffusion operator on $M$
which is symmetric with respect to a volume element $\vf^2\dv$,
where $\vf$ is a positive smooth function on $M$
and $\dv$ the volume element of the Riemannian metric on $M$ associated to $L$.
We assume that $L$ and $\vf$ are $\Gamma$-invariant.
Without loss of generality, we also assume that the action of $\Gamma$ on $M$ is effective.

We say that $M$ is \emph{$L$-recurrent mod $\Gamma$}
if the stochastic process on $\Gamma\backslash M$ induced by the $L$-diffusion on $M$ is recurrent.
Clearly, $L$-recurrence mod $\Gamma$ is equivalent to the property
that any $\Gamma$-invariant neighborhood of any orbit $\Gamma x$, $x\in M$,
is recurrent with respect to the $L$-diffusion.

We assume from now on now that $M$ is $L$-recurrent mod $\Gamma$.
We let $x_0\in M$ be a point with trivial isotropy group and $X=\Gamma x_0$ be its $\Gamma$-orbit.
\begin{enumerate}[label=(A)]
\item\label{afv}
We let $F_0\subseteq V_0$ be connected neighborhoods of $x_0$
such that $F_0$ is compact, $V_0$ is open and relatively compact,
and $\gamma F_0\cap V_0=\emptyset$ for all $\gamma\in\Gamma$ not equal to $e$.
For $x=\gamma x_0$, we let $F_x=\gamma F_0$ and $V_x=\gamma V_0$.
\end{enumerate}
Except for Assertion \ref{lsag}, which is an immediate consequence of \cref{lsm2},
the proof of the following lemma is very similar to the proof of \cref{lsc}
and will therefore be omitted.

\begin{lem}\label{lsa}
If $M$ is $L$-recurrent mod $\Gamma$, then
\begin{enumerate}
\item\label{lsaa}
any family $(F_x,V_x)$ satisfying \ref{afv} constitutes $LS$-data for $X$,
and the associated $LS$-measures are $\Gamma$-equivariant.
\item\label{lsab}
$X$ admits balanced $LS$-data satisfying \ref{afv}.
\item\label{lsag}
under the identification $X\cong\Gamma x_0$ via the orbit map,
the measures $\mu_x$ correspond to the left translates of a probability measure $\nu$ on $\Gamma$
with $\supp\nu=\Gamma$.
\end{enumerate}
Moreover, if the action of $\Gamma$ on $M$ is uniform, then $X$ is $*$-uniform and, in particular, admits uniform LS-data satisfying \ref{afv}.
\end{lem}

The proofs of the next assertions are very similar to the ones
of Theorems \ref{poha} -- \ref{trec} at the end of \cref{secdico}
and will therefore be omitted.

\begin{theadp}\label{pohag}
The assertions of Theorems \ref{poha} and \ref{mami} hold if $\Gamma\backslash M$ is compact,
the assertions of Theorems \ref{boha} and  \ref{trec} if $M$ is $L$-recurrent mod $\Gamma$. 
\end{theadp}

We now come to the extensions of Theorems \ref{dichot}, \ref{fchcin}, and \ref{noram}.

\setcounter{alphp}{4}
\begin{theop}\label{dichota}
If $M$ is $L$-recurrent mod $\Gamma$ and $h$ is a minimal positive $L$-harmonic function on $M$,
then $h$ is constant or $\gamma^*h/h$ is unbounded for some $\gamma\in\Gamma$.
\end{theop}

\begin{proof}
If there is a constant $b(\gamma)$ such that $\gamma^*h/h\le b(\gamma)$,
then $\gamma^*h/h=c(\gamma)$ for some constant $c(\gamma)$.
Hence, if there is no $\gamma\in\Gamma$ such that $\gamma^*h/h$ is unbounded,
there is a map $c\colon\Gamma\to\R^+$ such that $\gamma^*h/h=c(\gamma)$ for all $\gamma\in\Gamma$.
Clearly, $c$ is a homomorphism.

We let $x_0\in M$ be a point with trivial isotropy group and normalize $h$ so that $h(x_0)=1$.
Then $h(\gamma x_0)=c(\gamma)$ for all $\gamma\in\Gamma$.

We may assume that $M$ is $L$-transient and choose balanced $LS$-data satisfying \ref{afv}.
Then the $LS$-measures on $X=\Gamma x_0$ are symmetric.
Setting $\nu(\gamma)=\mu_{x_0}(\gamma x_0)$,
we obtain a symmetric probability measure $\nu$ on $\Gamma$.

Since $\ker c$ contains all elements of $\Gamma$ of order two,
we can write $\Gamma\setminus\ker c=B\cup B^{-1}$ as a disjoint union.
Using that $h|_X$ is $\mu$-superharmonic, by \cref{lsh3}, and that $c$ is a homomorphism,
we get that
\begin{align*}
	0 &= c(1)-1 \\
	&\ge\nu(c)-1 \\
	&= \sum_{\gamma\in\Gamma} \nu(\gamma)(c(\gamma)-1) \\
	&= \sum_{\gamma\in B} (\nu(\gamma)c(\gamma)+\nu(\gamma^{-1})c(\gamma^{-1})-2) \\
	&= \sum_{\gamma\in B} \nu(\gamma)(c(\gamma)+c(\gamma)^{-1}-2)
	\ge 0
\end{align*}
and hence that $c=1$.
Therefore $h$ is constant on $X=\Gamma x_0$.
It now follows from \cref{lsh3} that $h$ attains a minimum along $X$.
But then $h$ is constant, by the maximum principle.
\end{proof}

\begin{theop}
If $M$ is $L$-recurrent mod $\Gamma$,
then any bounded $L$-harmonic function on $M$ is invariant under the $FC$-hypercenter of $\Gamma$.
In particular, if $\Gamma$ is $FC$-hypercentral, then any bounded $L$-harmonic function on $M$ is constant.
\end{theop}

\begin{proof}[Proof of \cref{fchcin}]
Choose $X=\Gamma x_0$ as above and $LS$-data satisfying \ref{afv}.
Then $X$ and the associated family $\mu$ of $LS$-measures are $\Gamma$-invariant.
Let $h$ be a bounded $L$-harmonic function on $M$.
Then $h|_X$ is $\mu$-harmonic and therefore invariant under the $FC$-hypercenter of $\Gamma$,
by \cref{fchc2}.
But then $h$ is also invariant under the $FC$-hypercenter of $\Gamma$, by \cref{boha}'.\ref{bohab}.
If  $\Gamma$ is $FC$-hypercentral, then $h$ is constant,
again by \cref{fchc2} and \cref{boha}'.\ref{bohab}.
\end{proof}

\begin{theop}\label{gamam}
If $M$ is $L$-recurrent mod $\Gamma$,
then there is a $\Gamma$-invariant bounded projection $L^\infty(M)\to\mathcal{H}^\infty(M,L)$.
In particular, if  all bounded $L$-harmonic functions on $M$ are constant,
then $\Gamma$ is amenable.
\end{theop}

\begin{proof}
For $f\in L^\infty(M)$, let $f'=A(\tilde{f}|_X)$,
where $A\colon L^\infty(X)\to\mathcal{H}^\infty(X,\mu)$ is the projection from \cref{maram} and $\tilde{f}(x) = E_x[f \circ X_1]$.
Then $f'$ is the restriction of a unique bounded $L$-harmonic function $Hf$ on $M$,
and $H$ is the required projection.
\end{proof}

\section{Random walks and harmonic functions}
\label{secrwg}

In what follows, we consider \emph{random walks}, that is, Markov chains on countable sets.
Let $\nu$ be a random walk on $Y$,
given by a family $(\nu_y)_{y\in Y}$ of probability measures on $Y$.
We let $(P_y)_{y\in Y}$ be the associated family of probability measures
on the sample space $\Omega=Y^{\N_0}$.
We say that $\nu$ is \emph{symmetric} if
\begin{align}
	\nu_y(z)=\nu_z(y) \quad\text{for all $y,z\in Y$.}
\end{align}
We say that $\nu$ is \emph{irreducible} if, for all $y,z\in Y$,
there exist $y_1,\dots,y_k\in Y$ such that
\begin{align}\label{nui}
  \nu_y(y_1)\nu_{y_1}(y_2)\dots\nu_{y_{k-1}}(y_k)\nu_{y_k}(z) > 0.
\end{align}
For $k\ge2$,
we define a family $\nu^k=(\nu^k_y)_{y\in Y}$ of probability measures on $Y$ by
\begin{align}\label{nuk}
  \nu^k_y(z) = \sum_{x\in Y}\nu_y(x)\nu^{k-1}_x(z),
\end{align}
where $\nu^1=\nu$.
Irreducibility is then equivalent to the property that, for all $y,z\in Y$,
\begin{align}
	\nu^k_y(z)>0 \quad\text{for some $k\ge1$.} 
\end{align}
If $\nu$ is symmetric, then also $\nu^k$, for any $k\ge2$.

We assume that a countable group $\Gamma$ acts on $Y$ and that $\nu$ is $\Gamma$-invariant,
that is, that we have
\begin{align}\label{gamin}
  \nu_{\gamma y}(\gamma z)=\nu_y(z)
\end{align}
for all $y,z\in Y$ and $\gamma\in\Gamma$.
Then the $\nu^k$ are also $\Gamma$-invariant.

For $X\subseteq Y$ and $\omega\in Y^{\N_0}$,
we let $R^X(\omega)=\inf\{k\ge1\mid\omega_k\in X\}$.
We say that $X$ is \emph{$\nu$-recurrent} if $P_y[R^X(\omega)<\infty]=1$ for any $y\in Y$.
Then the family $\mu=(\mu_y)_{y\in Y}$ of hitting probabilities,
 \begin{align}\label{hitx}
	\mu_y(x)
	= P_y[\omega(R^X(\omega))=x]
	= \sum_{i\ge0}\sum_{y_1,\dots,y_{i}\in Y\setminus X}\nu_y(y_1)\nu_{y_1}(y_2)\dots\nu_{y_{i}}(x),
\end{align}
where $y\in Y$ and $x\in X$, are probability measures on $X$
and define a random walk with sample space $\Omega=Y\times X^\N$.
By the definition of $\mu$, we have
\begin{align}\label{hitx2}
  \mu_y(x) = \nu_y(x) + \sum_{z\in Y\setminus X} \nu_y(z)\mu_z(x)
\end{align}
for all $y\in Y$ and $x\in X$.
By the $\nu$-recurrence of $X$,
the space of sequences in $Y^{\N_0}$ which have infinitely many of its members in $X$  
has full $P_y$-measure for all $y\in Y$,
and the subsequences consisting of the corresponding starting members and members belonging to $X$ will be called the \emph{$X$-subsequences}.

\begin{prop}\label{balmu}
Let $X\subseteq Y$ be a $\Gamma$-invariant $\nu$-recurrent subset.
Then the family $\mu$ of hitting probabilities has the following properties:
\begin{enumerate}
\item\label{balmui}
$\mu$ is $\Gamma$-invariant.
\item\label{balmus} 
$\mu$ is symmetric on $X$ if $\nu$ is symmetric.
\item\label{balmug}
The Green functions $g$ of $\mu$ and $G$ of $\nu$ satisfy
\begin{align*}
	g(y,x)=G(y,x) \quad\text{for all $y\in Y$ and $x\in X$.}
\end{align*}
In particular,
the $\mu$-random walk on $X$ is transient if and only if the $\nu$-random walk on $Y$ is transient.
\end{enumerate}
\end{prop}

\begin{proof}
\eqref{balmui} is clear,
and \eqref{balmus} follows immediately from \eqref{hitx}.
As for \eqref{balmug}, we have
\begin{align*}
	G(y,x)
	&= \delta_y(x) + \sum_{i\ge1} P_y[\omega_i=x] \\
	&= \delta_y(x) + \sum_{i\ge0}\sum_{y_1,\dots,y_{i}\in Y}\nu_y(y_1)\nu_{y_1}(y_2)\dots\nu_{y_{i}}(x) \\ 
	&= \delta_y(x) + \sum_{j\ge0}\sum_{x_1,\dots,x_{j}\in X}\mu_y(x_1)\mu_{x_1}(x_2)\dots\mu_{x_{j}}(x) \\ 
	&= g(y,x),
\end{align*}
where we use \eqref{hitx} to pass from the second to the third line.
\end{proof}

We say that the action of $\Gamma$ on $Y$ is
\emph{cofinite} if $\Gamma\backslash Y$ is finite.
We say that $Y$ is \emph{$\nu$-recurrent mod $\Gamma$} if the orbits of $\Gamma$ in $Y$ are $\nu$-recurrent.

Recall that a function $h$ on $X$ is \emph{$\mu$-harmonic} if it satisfies \eqref{muharm} for all $y\in X$.
Similarly, a function $h$ on $Y$ is \emph{$\nu$-harmonic} if it satisfies \eqref{muharm} for all $y\in Y$,
where $Y$ is substituted for $X$ and $\nu$ for $\mu$.
If $h$ is $\mu$-harmonic, then $h$ is $\mu^k$-harmonic for all $k\ge1$,
and similarly for $\nu$-harmonic functions.

\begin{theadpp}\label{pohar}
Substituting $Y$ for $M$ and $\nu$ for $L$ and letting $X$ be an orbit of $\Gamma$ in $Y$,
the assertions of Theorems \ref{boha} and \ref{trec} hold if $Y$ is $\nu$-recurrent mod $\Gamma$.
Here $LS$-path discretization in \cref{trec}.\ref{recp} has to be replaced by passage to $X$-subsequences.
\end{theadpp}

\begin{proof}
Given a $\mu$-harmonic function $h$ on $X$,
we want to show that it is the restriction of a $\nu$-harmonic function on $Y$.
To that end, we extend $h$ to $Y$ by setting
\begin{align}\label{hext}
h(y) = \sum_{x\in X} \mu_y(x)h(x). 
\end{align}
Since $h$ is $\mu$-harmonic on $X$, the extension agrees with the original $h$ on $X$.
However, we have to verify that $h$ is well defined on $Y\setminus X$, that is,
that the sum on the right hand side is finite for all $y\in Y\setminus X$.
Since $\mu_y$ is a probability measure on $X$ for all $y\in Y$,
this is clear if $h$ is bounded.

\begin{lem}\label{muest}
For any $y\in Y$, there exist $y_0\in X$ and a constant $c>0$ such that
\begin{align*}
	\mu_y(x)\le c\mu_{y_0}(x) \quad\text{for all $x\in X$.}
\end{align*}
\end{lem}

\begin{proof}
We can assume $y\in Y\setminus X$.
By the irreducibility of $\nu$,
there exist $i\ge0$, $y_0\in X$, and $y_1,\dots,y_i\in Y\setminus X$
such that $\nu_{y_0}(y_{1})\cdots\nu_{y_{i}}(y)>0$.
Using \eqref{hitx2}, we get
\begin{align*}
  \nu_{y_i}(y)\mu_{y}(x)
  \le \sum_{z\in Y\setminus X} \nu_{y_i}(z)\mu_{z}(x)
  = \mu_{y_i}(x) - \nu_{y_i}(x)
  \le \mu_{y_i}(x)
\end{align*}	
for all $x\in X$.	
Hence $\mu_{y}(x)\le\mu_{y_i}(x)/\nu_{y_i}(y)$.
Recursively, we get that
\begin{align}\label{muest2}
	\mu_{y}(x)\le\frac1{\nu_{y_i}(y)\cdots\nu_{y_0}(y_{1})}\mu_{y_0}(x)
\end{align}
for all $x\in X$.
\end{proof}

We continue with the proof of the theorem.
Therefore
\begin{align*}
	\sum_{z\in Y} \nu_y(z)h(z)
	&=  \sum_{x\in X.z\in Y} \nu_y(z)\mu_z(x)h(x) \\
	&= \sum_{x,z\in X} \nu_y(z)\mu_z(x)h(x) + \sum_{x\in X} (\mu_y(x)-\nu_y(x))h(x) \\
	&= \sum_{x\in X} \mu_y(x)h(x)
	= h(y),
\end{align*}
and hence the extension \eqref{hext} of $h$ is $\nu$-harmonic on $Y$. It follows by arguing as in Lemma \ref{muhar} that this is the unique positive $\nu$-harmonic extension of $h$, which we will denote by $Eh$. Keeping in mind the above modifications, the proofs of the analogs are similar to the proofs of the original ones.
\end{proof}

\begin{theopp}\label{discrweep}
Substituting $Y$ for $M$ and $\nu$ for $L$ and letting $X$ be an orbit of $\Gamma$ in $Y$, the assertions of Theorem \ref{poha} hold if $\Gamma\backslash Y$ is finite.
\end{theopp}

\begin{proof}
	It suffices to prove that $h(y) = \mu_y(h)$ for any $h \in \mathcal{H}^+(Y,\nu)$ and $y \in Y$. Let $h \in \mathcal{H}^+(Y,\nu)$. In view of (\ref{hitx2}), it is sufficient to show that $h(y) = \mu_y(h)$ for any $y \in Y \setminus X$. Choose a finite set $A$ of representatives of the $\Gamma$-orbits in $Y$ that do not intersect $X$. Since $X$ is recurrent and $A$ is finite, there exist $c > 0$ and $N \in \mathbb{N}$ such that for any $y \in A$ there exists $n(y) \leq N$, $y =: y_0,\dots,y_{n(y) -1} \in Y \setminus X$, and $\bar{y} := y_{n(y)} \in X$ such that
	\[
	\nu_{y}(y_1)\nu_{y_1}(y_2) \dots \nu_{y_{n(y) - 1}}(\bar{y}) \geq c.
	\]
	Since $\nu$ is irreducible and $A$ is finite, there exists $c^\prime > 0$ such that for any $y \in A$ there exists $k \in \mathbb{N}$ satisfying $\nu_{\bar{y}}^{k}(y) \geq c^\prime$.
	Keeping in mind that $\nu$-harmonic functions are also $\nu^{k}$-harmonic, we derive that $\varphi(\bar{y}) \geq c^\prime \varphi(y)$ for any $y \in A$ and any $\varphi \in \mathcal{H}^+(Y,\nu)$. Since $\nu$ is $\Gamma$-invariant, we readily see that $h_g = h \circ g$ is $\nu$-harmonic for any $g \in \Gamma$, and therefore, $h(g\bar{y}) \geq c^\prime h(g y)$ for any $y \in A$ and $g \in \Gamma$.
	
	For any $y \in Y\setminus X$, using that $h$ is $\nu$-harmonic recursively, we obtain that
	\begin{align*}
		h(y) &= \sum_{y_1 \in Y} \nu_y(y_1) h(y_1) = \sum_{y_1 \in X} \nu_y(y_1) h(y_1) + \sum_{y_1 \in Y \setminus X } \nu_y(y_1) h(y_1) \\
		&= \sum_{y_1 \in X} \nu_y(y_1) h(y_1) + \sum_{y_2 \in Y, y_1 \in Y \setminus X } \nu_y(y_1) \nu_{y_1}(y_2) h(y_2) = Q_n(y) + E_n(y)
	\end{align*}
	for any $n \in \mathbb{N}$, where
	\begin{align*}
		Q_n(y) &= \sum_{i=1}^n \sum_{y_i \in X , y_{i-1},\dots,y_1 \in Y \setminus X} \nu_y(y_1) \nu_{y_1}(y_2) \dots \nu_{y_{i-1}}(y_i) h(y_i),\\
		E_n(y) &= \sum_{y_n,\dots,y_1 \in Y \setminus X} \nu_y(y_1) \nu_{y_1}(y_2) \dots \nu_{y_{n-1}}(y_n) h(y_n).
	\end{align*}
	It is evident that $Q_n(y)$ is non-decreasing with respect to $n$, and thus, $E_n(y)$ is non-increasing with respect to $n$.
	From the definition of $\mu_y$, it suffices to prove that $E_n(y) \rightarrow 0$ for any $y \in Y \setminus X$.
	
	Bearing in mind that $y = g z$ for some $g \in \Gamma$ and $z \in A$, and that $N \geq n(z)$, we deduce that
	\[
	Q_{N}(y) \geq \nu_{gz}(gz_1)\dots\nu_{gz_{n(z)-1}}(g\bar{z}) h(g\bar{z}) \geq ch(g \bar{z}) \geq c c^\prime h(g z) = c c^\prime h(y).
	\]
	This, together with $h(y) = Q_N(y) + E_N(y)$, gives the estimate
	\[
	E_{N}(y) \leq (1-cc^\prime) h(y)
	\]
	for any $y \in Y \setminus X$. For any $y \in Y \setminus X$ and $k \in \mathbb{N}$, we compute
	\begin{eqnarray}
		E_{(k+1)N}(y) &=& \sum_{y_{(k+1)N},\dots,y_1 \in Y \setminus X} \nu_{y}(y_1) \dots \nu_{y_{(k + 1)N -1}}(y_{(k+1)N}) h(y_{(k+1)N}) \nonumber\\
		&=& \sum_{y_{kN},\dots,y_{1} \in Y \setminus X} \nu_{y}(y_1) \dots \nu_{y_{k N -1}}(y_{kN}) E_{N}(y_{kN}) \nonumber\\
		&\leq& (1-cc^\prime) E_{kN}(y), \nonumber
	\end{eqnarray}
	which yields that $E_{kN}(y) \leq (1-cc^\prime)^{k} h(y)$ for any $k \in \mathbb{N}$. Since $(E_n(y))_{n \in \mathbb{N}}$ is non-increasing, we conclude that $E_n(y) \rightarrow 0$ for any $y \in Y \setminus X$, as we wished.
\end{proof}

\begin{lem}\label{minMartin}
	If $\Gamma \backslash Y$ is finite and the $\nu$-random walk on $Y$ is transient, then $\partial_{\nu}^{\min} Y \subseteq \partial_\nu X$.
\end{lem}

\begin{proof}
	Let $\xi \in \partial_{\nu}^{\min} Y$. Then $K(\cdot,\xi) \in H^+(Y,\nu)$ and Theorem \ref{discrweep} yields that $R K (\cdot,\xi) \in H^+(X,\mu)$. As in the proof of Theorem \ref{mfx}, minimality of $K(\cdot,\xi)$ implies the minimality of $RK(\cdot,\xi)$. This yields that $R K(\cdot ,\xi) = k(\cdot,\eta)$ for some $\eta \in \partial_{\mu}^{\min}X$.
	We derive from Theorem D''(2) (that is, the analog of \cref{trec}(\ref{recm}))
	that there exists $\xi^\prime \in \partial_{\nu} X \cap H^+(Y,\nu)$ such that $RK(\cdot,\xi^\prime) = k(\cdot,\eta)$. Combining these, we deduce that $K(\cdot, \xi^\prime) = Ek(\cdot,\eta) = K(\cdot,\xi)$, which means that $\xi = \xi^\prime \in \partial_{\nu}X$.
\end{proof}

The analog of Theorem \ref{mami} in the setting of random walks is the following:

\setcounter{alphpp}{1}
\begin{theopp}
Suppose that $\Gamma \backslash Y$ is finite and that $\nu$-random walk on $Y$ is transient. Then:
\begin{enumerate}
	\item the $\mu$-random walk on $X$ is transient and  restriction of Martin kernels
	\begin{align*}
		R\colon\partial_\nu^{\min} Y \to \partial_\mu^{\min} X
	\end{align*}
	is a $\Gamma$-equivariant homeomorphism with inverse $E$.
	\item $\partial_{\nu} X = \partial_{\nu}^{\min} Y$ if and only if $\partial_{\mu} X = \partial_{\mu}^{\min} X$, and then $\partial_{\nu} Y \cap H^{+}(Y,\nu) = \partial_{\nu}^{\min} Y$
\end{enumerate}
\end{theopp}

\begin{proof}
Choose origin $x_0 \in X$ for Martin kernels associated to the $\nu$-random walk on $Y$ and the $\mu$-random walk on $X$.
	We know from Theorems \ref{discrweep} and D''(2) that	
	\[
	R\colon\partial_\nu X \cap \mathcal H^+(Y,\nu) \to \partial_\mu X\cap\mathcal H^+(X,\mu)
	\]
	is a $\Gamma$-equivariant homeomorphism with inverse $E$.
	As in the proof of Theorem \ref{mfx}, $R$ maps minimal positive $\nu$-harmonic functions of $Y$ surjectively to minimal positive $\mu$-harmonic functions on $X$.
	This means that $R$ restricts to a homeomorphism
	\[
	R \colon \partial_{\nu}^{\min} Y \cap \bar{X} \to \partial_{\mu}^{\min} X.
	\]
	The first part of the proof is completed by Lemma \ref{minMartin}.
	
	We know from Theorem D''(2) that restriction $R \colon \partial_{\nu} X \to \partial_{\mu}X$ is surjective. In view of (1), it is easily checked that $R^{-1}(\partial_{\mu}^{\min}X) = \partial_{\nu}^{\min}Y$. This readily implies the equivalence of the first two statements of (2).
	
	Assume now that $\partial_{\mu} X = \partial_{\mu}^{\min} X$. Let $\xi \in \partial_{\nu} Y \cap \mathcal{H}^+(Y,\nu)$ and consider a sequence $(y_n)_{n \in \mathbb{N}}$ in $Y$ with $y_n \rightarrow \xi$ in $\mathcal{M}(Y,\nu)$.
	Fix a finite set $A$ of representatives of the $\Gamma$-orbits in $Y$.
	After passing to a subsequence if necessary, we may suppose that there exist $z \in A$ and $(g_n)_{n \in \mathbb{N}} \subseteq \Gamma$ such that $y_n = g_n z$ for any $n \in \mathbb{N}$.
	Since the random walk on $Y$ induced by $\nu$ is irreducible, we know that there exists $x_0 \in A \cap X$, $k_1, k_2 \in \mathbb{N}$ and $c > 0$ such that $\nu^{k_1}_{y}(x_0) \geq c$ and $\nu^{k_2}_{x_0}(y) \geq c$. Considering the divergent sequence $x_n = g_n x_0$ of $X$, we have that $\nu_{y_n}^{k_1}(x_n) \geq c$ and $\nu_{x_n}^{k_2}(y_n) \geq c$ for any $n \in \mathbb{N}$, which yields that the Martin kernels satisfy $K(\cdot , y_n) \leq c^{-2} K(\cdot , x_n)$ pointwise, for sufficiently large $n \in \mathbb{N}$. It is evident that the same holds for their restrictions to $X$. After passing to a subsequence if necessary, we may assume that $x_n \rightarrow \eta \in \partial_{\mu} X$ in $\mathcal{M}(X,\mu)$. By assumption, $k(\cdot,\eta)$ is a minimal positive $\mu$-harmonic function. Keeping in mind that $k(\cdot,y_n) \rightarrow R K(\cdot , \xi) \in \mathcal{H}^+(X,\mu)$, by virtue of Theorem \ref{discrweep}, it follows from the aforementioned estimate that $k(\cdot,\eta) = R K(\cdot,\xi)$. We conclude that $K(\cdot,\xi) = E k(\cdot , \eta)$ which is a minimal positive $\nu$-harmonic function, and thus $\xi \in \partial_{\nu}^{\min} Y$.
\end{proof}

\begin{lem}\label{gamhom2}
If $\Gamma\backslash X$ is finite and $h$ is a positive $\mu$-harmonic function on $X$,
then $\gamma^*h/h$ is bounded for any $\gamma$ in the center of $\Gamma$.
\end{lem}

\begin{proof}
Let $\gamma$ belong to the center of $\Gamma$.
For all $a\in\Gamma$, $x\in X$, and $k\ge1$, we then have
\begin{align*}
  \mu^k_{ax}(\gamma ax)=\mu^k_{ax}(a\gamma x) = \mu^k_x(\gamma x).
\end{align*}
Hence if $c>0$ and $k\ge1$ are such that $\mu^k_x(\gamma x)\ge c$
for all $x$ in a finite set of representatives of the orbits of $\Gamma$ in $X$,
then $\mu^k_x(\gamma x)\ge c>0$ for all $x\in X$.
With such $c$ and $k$, we obtain
\begin{align*}
  h(x)
  = \sum_{y\in X} \mu^k_x(y)h(y)
  \ge \mu^k_x(\gamma x)h(\gamma x)
  \ge c h(\gamma x)
\end{align*}
for any $x\in X$, and then $\gamma^*h/h\le1/c$.
\end{proof}

Together with the $\Gamma$-invariance, symmetry of $\nu$ gives
 \begin{align}\label{musym}
   \nu_x(\gamma^{-1}x) = \nu_{\gamma x}(x) = \nu_x(\gamma x),
 \end{align}
 for all $x\in Y$ (and similarly for $\mu$ on $X$).

\setcounter{alphpp}{4}
\begin{theopp}\label{dichot2}
Suppose that $\nu$ is symmetric and that $Y$ is $\nu$-recurrent mod $\Gamma$.
Let $h$ be a minimal positive $\nu$-harmonic function on $Y$.
Then either $h$ is constant or there is a $\gamma\in\Gamma$ such that $\gamma^*h/h$ is unbounded.
\end{theopp}

\begin{proof}
Let $x_0\in Y$ and set $X=\Gamma x_0$.
Normalize $h$ so that $h(x_0)=1$
and assume that $\gamma^*h/h$ is bounded for all $\gamma\in\Gamma$.
Then $c(\gamma)=\gamma^*h/h$ is a constant for all $\gamma\in\Gamma$.
Clearly, $c\colon\Gamma\to\R_{>0}$ is a homomorphism.
Using that $h|_X$ is $\mu$-superharmonic (by an analog of \cref{lsh3}) and \eqref{musym}, we get
\begin{align*}
	0 &= h(x_0)-1 \\
	&\ge \mu_{x_0}(h)-1 \\
	&= \sum_{x\in X} \mu_{x_0}(x)(h(x)-1) \\
	&= \frac1{2|\Gamma_0|}\sum_{\gamma\in\Gamma} (\mu_{x_0}(\gamma x_0)h(\gamma x_0)+\mu_{x_0}(\gamma^{-1}x_0)h(\gamma^{-1}x_0)-2) \\
	&= \frac1{2|\Gamma_0|}\sum_{\gamma\in\Gamma} \mu_{x_0}(\gamma x_0)(c(\gamma)+c(\gamma)^{-1}-2)
	\ge 0,
\end{align*}
where $\Gamma_0$ denotes the isotropy group of $x_0$.
It follows that $c(\gamma)=1$ for all $\gamma\in\Gamma$ with $\mu_{x_0}(\gamma x_0)>0$.
By passing to $\mu^k$ for an appropriate $k$ if necessary,
the latter can be achieved for each $\gamma\in\Gamma$.
Hence $h$ is invariant under $\Gamma$.
But then $h$ is constant since $Y$ is $\mu$-recurrent mod $\Gamma$.
\end{proof}

If $\Gamma$ is Abelian and the action of $\Gamma$ on $Y$ is cofinite,
then $\gamma^*h/h$ is bounded for all $\gamma\in\Gamma$, by \cref{gamhom2}.
Hence \cref{dichot2} has the following consequence.

\begin{cor}\label{dichot4}
If $\Gamma$ is Abelian, $\nu$ is symmetric, and the action of $\Gamma$ on $Y$ is cofinite
then any positive $\nu$-harmonic function on $Y$ is constant.
\end{cor}

For left-invariant Markov chains on groups,
\cref{dichot4} is well known and attributed to Choquet and Deny,
who characterize, more generally, solutions $\mu$ of the equation $\nu=\mu*\nu$
for general measures $\nu$ on locally compact Abelian groups \cite[Theorem 3]{CD}.

\begin{exa}
The function $e^x$ is a positive $\nu$-harmonic function on $\Z$
with respect to the non-symmetric probability measure $\nu$ supported on $\{-1,1\}$
with $\nu(-1)=e/(e+1)$ and $\nu(1)=1/(e+1)$.
\end{exa}

Using Margulis's reduction to the Abelian case \cite{Ma}, we have the following consequence of \cref{dichot4}.

\begin{cor}\label{nil2}
If $\Gamma$ is nilpotent and $\mu$ an irreducible symmetric probability measure on $\Gamma$,
then any positive $\mu$-harmonic function on $\Gamma$ is constant.
\end{cor}

We now come to bounded harmonic functions.
We will need the following special case of \cite[Theorem 3.9]{Li}.

\begin{lem}\label{fchc4}
Assume that $\Gamma\backslash Y$ is finite.
Let $\sigma\in\Gamma$ and $h\in\mathcal{H}^\infty(Y,\mu)$.
Assume that $h(\gamma\sigma y)=h(\sigma\gamma y)$
for all $\gamma\in\Gamma$ and $y\in Y$.
Then $h(\sigma y)=h(y)$ for all $y\in Y$.
\end{lem}

\begin{theopp}\label{fchc2}
If $\Gamma\backslash Y$ is finite,
then all bounded $\nu$-harmonic functions on $Y$ are invariant under the $FC$-hypercenter of $\Gamma$.
In particular, if $\Gamma$ is $FC$-hypercentral,
then $Y$ does not have non-constant bounded $\nu$-harmonic functions.
\end{theopp}

The arguments in the proof are taken from the proofs of Lemma 2.4 and Corollary 2.5 in \cite{LZ}.

\begin{proof}[Proof of \cref{fchc2}]
Let $N$ be a normal subgroup of $\Gamma$ such that all $h\in\mathcal{H}^\infty(Y,\nu)$ are $N$-invariant and $\gamma\in\Gamma$ have the property that its conjugacy class in $\Gamma/N$ is finite.
Then the centralizer of $\gamma$ in $\Gamma/N$ has finite index in $\Gamma/N$
and, therefore, its preimage $\Gamma'$ in $\Gamma$ finite index in $\Gamma$.
We have $\gamma\in\Gamma'$ and $N\subseteq\Gamma'$.
Now $\Gamma'\backslash Y$ is finite since $\Gamma'$ has finite index in $\Gamma$.
By definition, $[\gamma,\Gamma']\subseteq N$,
hence any commutator $\gamma\sigma\gamma^{-1}\sigma^{-1}$ with $\sigma\in\Gamma'$ leaves any $h\in\mathcal{H}^\infty(Y,\nu)$ invariant.
But then $\gamma$ leaves all $h\in\mathcal{H}^\infty(Y,\nu)$ invariant, by \cref{fchc4}.

Let $N$ now be the normal group of $\gamma\in\Gamma$ which leave all $h\in\mathcal{H}^\infty(Y,\nu)$ invariant.
Then, by the above, $FC(\Gamma/N)$ is trivial.
Hence $\Gamma_\alpha\subseteq N$ for any ordinal $\alpha$.

If $\Gamma$ is hypercentral,
then $h$ is $\Gamma$-invariant and has therefore at most $|\Gamma\backslash Y|$ many values.
In particular, $h$ has a maximum.
Hence $h$ is constant, by the maximum principle.
\end{proof}

\begin{theopp}\label{maram}
There is a $\Gamma$-invariant bounded projection $L^\infty(X)\to\mathcal{H}^\infty(X,\mu)$.
In particular, if the action of $\Gamma$ is proper and all bounded $\mu$-harmonic functions on $X$ are constant,
then $\Gamma$ is amenable.
\end{theopp}

The proof of \cref{maram} consists of a translation of the proof of \cite[Theorem 3']{LS}
to the discrete case.

\begin{proof}[Proof of \cref{maram}]
Let $m\colon L^\infty(\N_0)\to\R$ be an invariant mean for the Abelian semigroup $\N_0$,
that is, $m$ is a functional with $m(1)=1$, $m(f)\ge0$ if $f\ge0$,
and $m(f_k)=m(f)$ for all $f\in L^\infty(\N_0)$ and $k\in\N_0$, where $f_k(l)=f(k+l)$.

Let $x\in X$ and $f\in L^\infty(X)$.
Then $f_x\colon\N_0\to\R$, $f_x(k)=\mu^k_x(f)$, belongs to $L^\infty(\N_0)$, and we set
\begin{align*}
	\hat m\colon L^\infty(X) \to L^\infty(X), \quad
	\hat m(f)(x) = m(f_x).
\end{align*}
Then $\|\hat m\|\le\|m\|=1$ and $\hat m(f)=f$ if $f$ is $\mu$-harmonic.
Moreover, $\hat m$ is $\Gamma$-invariant, since $\mu$ is $\Gamma$-invariant.
Furthermore,
\begin{align*}
	\hat m(\mu(f))(x) = m((\mu(f))_x)
	&= m((k\mapsto\mu^{k+1}_x(f))) \\
	&= m((f_x)_1) = m(f_x) = \hat m(f)(x)
\end{align*}
since $m$ is an invariant mean.
Now
\begin{align*}
	\mu(f)_x(k) &= \sum_{y,z\in X} \mu^k_x(y)\mu_y(z)f(z) \\
	&= \sum_{y,z\in X} \mu_x(y)\mu^k_y(z) f(z) \\
	&= \sum_{y\in X}\mu_x(y)f_y(k).
\end{align*}
For an exhausting sequence $F_1\subseteq F_2\subseteq\dots$ of finite subsets of $X$,
we then obtain
\begin{align*}
	|\mu(f)_x(k) - \sum_{y\in F_n}\mu_x(y)f_y(k)|
	\le \sum_{y\in X\setminus F_n}\mu_x(y) \|\mu^k(f)\|_\infty
	\le \sum_{y\in X\setminus F_n}\mu_x(y) \|f\|_\infty
\end{align*}
for all $k\in\N_0$.
Therefore
\begin{align*}
	\|\mu(f)_x - \sum_{y\in F_n}\mu_x(y)f_y\| \le \sum_{y\in X\setminus F_n}\mu_x(y) \|f\|_\infty.
\end{align*}
Since
\begin{align*}
	| \hat m(f)(x) - \sum_{y\in F_n}\mu_x(y)\hat m(f)(y) |
	&= |m(\mu(f)_x) - \sum_{y\in F_n}\mu_x(y)m(f_y) | \\
	&= |m(\mu(f)_x) - m(\sum_{y\in F_n}\mu_x(y)f_y) | \\	
	&\le \|m\| \|\mu(f)_x - \sum_{y\in F_n}\mu_x(y)f_y \| \\
	&\le \|m\| \sum_{y\in X\setminus F_n}\mu_x(y) \|f\|_\infty,
\end{align*}
which tends to $0$ as $n\to\infty$, we conclude that
\begin{align*}
	\hat m(\mu(f)_x) = \sum_{y\in X} \mu_x(f)\hat m(f)(y),  
\end{align*}
that is, $\hat m(f)$ is a $\mu$-harmonic function on $X$.
The first assertion follows.

If the action of $\Gamma$ is proper, then the isotropy groups $\Gamma_x$ of $\Gamma$ are finite.
Fixing a set $R\subseteq X$ of representatives of the $\Gamma$-orbits in $X$,
we define
\begin{align*}
	E \colon L^\infty(\Gamma) \to L^\infty(X), \quad
	E(f)(\gamma x) = \frac1{|\Gamma_x|}\sum_{\sigma\in\Gamma_x}f(\gamma\sigma),
\end{align*}
where $x\in R$.
Then $\|E\|\le1$ and $E(1)=1$.
Furthermore, $E$ is $\Gamma$-equivariant,
\begin{align*}
	E(\tau^*f)(\gamma x)
	&= \frac1{|\Gamma_x|}\sum_{\sigma\in\Gamma_x}(\tau^*f)(\gamma\sigma)\\
	&= \frac1{|\Gamma_x|}\sum_{\sigma\in\Gamma_x}f(\tau\gamma\sigma) \\
	&= E(f)(\tau\gamma x) 
	= (\tau^*E(f))(\gamma x).
\end{align*}
Hence $\hat m$ induces an invariant mean on $\Gamma$ if $\mathcal{H}^\infty(X,\mu)$ is trivial.
\end{proof}

\begin{rem}
For the second assertion of \cref{maram},
we only need that the stabilizers of the $\Gamma$-action on $X$ are amenable.
Then we would set $E(f)(\gamma x)=m_x(\gamma^*f)$,
where $x\in R$ and $m_x$ is an invariant mean for $\Gamma_x$.
\end{rem}



\end{document}